\documentclass[11pt]{amsart}

\usepackage{amsfonts,amssymb,amsmath,amsthm} 
\usepackage{mathrsfs}   
\usepackage{dsfont}     

\usepackage{graphicx}    
\usepackage{epstopdf}    
\usepackage{subcaption}  
\usepackage{booktabs}    
\usepackage{tabularx}    
\usepackage{tikz}        
\usepackage{multirow}    

\usepackage{algorithm}   
\usepackage{algpseudocode}

\usepackage{url}         
\usepackage[colorlinks, citecolor=red]{hyperref}

\newtheorem{definition}{Definition}[section]

\newtheorem{theorem}[definition]{Theorem}
\newtheorem{lemma}[definition]{Lemma}

\newtheorem{remark}[definition]{Remark}

\makeatletter
\let\AM@caption\caption
\let\caption\relax
\let\AM@subcaption\@subcaption
\makeatother

\date{}
\begin{document}
	\baselineskip 18pt
	\title[RBKVS]
	{Randomized block Kaczmarz with volume sampling:  Momentum acceleration and efficient implementation}
	\author{Ruike Xiang}
	\address{School of Mathematical Sciences, Beihang University, Beijing, 100191, China. }
	\email{xiangrk@buaa.edu.cn}

	\author{Jiaxin Xie}
	\address{LMIB of the Ministry of Education, School of Mathematical Sciences, Beihang University, Beijing, 100191, China. }
	\email{xiejx@buaa.edu.cn}
	
		\author{Qiye Zhang}
	\address{LMIB of the Ministry of Education, School of Mathematical Sciences, Beihang University, Beijing, 100191, China. }
	\email{zhangqiye@buaa.edu.cn}

	\begin{abstract}
	The randomized block Kaczmarz (RBK) method is a widely utilized iterative scheme for solving large-scale linear systems. However, the theoretical analysis and practical effectiveness of this method heavily rely on  a good  row paving of the coefficient matrix. This motivates us to introduce a novel block selection strategy to the RBK method, called volume sampling,  in which the probability of selection is proportional to the volume spanned by the rows of the selected submatrix. To further enhance the practical performance, we develop and analyze a momentum variant of the method. Convergence results are established and demonstrate the notable improvements in convergence factor of the RBK method brought by the volume sampling and the momentum acceleration. Furthermore, to  efficiently implement the RBK method with volume sampling, we propose  an efficient algorithm that enables volume sampling from a sparse matrix with sampling complexity that is only logarithmic in dimension. Numerical experiments confirm our theoretical results.
\end{abstract}

\maketitle

\let\thefootnote\relax

\let\thefootnote\relax

\section{Introduction}

The Kaczmarz method \cite{karczmarz1937angenaherte} and its randomized variant \cite{strohmer2009randomized} are very popular iteration solvers for solving the large-scale linear systems
\begin{equation}\label{eq1}
	Ax=b, A\in\mathbb{R}^{m\times n}, b\in\mathbb{R}^m.
\end{equation}
For any $i\in [m] := \{1,\ldots, m\}$, let $A_{i }$ denote the $i$-th row of $A$ and $b_i$ denote the $i$-th entry of $b$. Starting from an arbitrary initial vector $x^0$, the method generates the subsequent iterate $x^{k+1}$ by
$$
x^{k+1}=x^k-\frac{A_{i_k }x^k-b_{i_k}}{\|A_{i_k }\|^2_2} A_{i_k }^\top,
$$
where $\top$  denotes the transpose, and $i_k$ is the row index  chosen according to various selection strategies, such as cyclic, greedy, or random.
Due to their simplicity and efficiency,  Kaczmarz-type methods have been widely used in many fields, such as signal processing  \cite{byrne2003unified}, image
reconstruction \cite{herman2008image,popa2004kaczmarz},  and computer tomography \cite{herman2009fundamentals,natterer2001mathematics}.

In some situations, one might prefer a block variant of the Kaczmarz method to efficiently tackle the linear system \eqref{eq1}. Inspired by the block-iterative scheme described in \cite{elfving1980block}, a block variant of the randomized Kaczmarz (RK) methods was first developed and analyzed by Needell and Tropp \cite{needell2014paved}.
The randomized block Kaczmarz (RBK) method initiates by partitioning the rows of the coefficient matrix into $t$ subsets, forming a partition $\Gamma=\{\mathcal{S}_1,\ldots,\mathcal{S}_t\}$. 
The iteration scheme of the RBK method is described as
\begin{equation}
	\label{rbk-gave}
	x^{k+1}=x^k-A^\dagger_{\mathcal{S}_k }(A_{\mathcal{S}_k }x^k-b_{\mathcal{S}_k}),
\end{equation}
where  $A_{\mathcal{S}_k }$ and $b_{\mathcal{S}_k}$ denote  the submatrix of $A$ and the subvector of $b$ indexed by $\mathcal{S}_k$, respectively, and $A^\dagger_{\mathcal{S}_k }$ is the (Moore-Penrose) pseudoinverse of $A_{\mathcal{S}_k }$. Particularly, if $t= m$, that is, if the number of blocks is equal to the number of rows, each block consists of a single row and we recover the RK method.

To analyze the convergence of the RBK method, it is essential to define specific quantities as outlined in \cite{needell2014paved}.
Let $\lambda_{\min}(\cdot)$ and $\lambda_{\max}(\cdot)$ denote the minimum and maximum eigenvalues of a certain matrix, respectively.
A row paving
$(t,\beta_1, \beta_2)$ of a matrix $A$ is a partition $\Gamma=\{\mathcal{S}_1,\ldots,\mathcal{S}_t\}$ of the set $[m]$ that  satisfies the following conditions for each subset $\mathcal{S}_i$ in the partition
$$
\beta_1\leq \lambda_{\min}(A_{\mathcal{S}_i }A^\top_{\mathcal{S}_i }) \ \ \text{and} \ \ \lambda_{\max}(A_{\mathcal{S}_i }A^\top_{\mathcal{S}_i })\leq \beta_2, \ \  \forall \ \mathcal{S}_i\in\Gamma, 
$$
where $t$ indicates the size of the paving,  and  $\beta_1$ and $\beta_2$ denote the lower and upper paving bounds, respectively. 
Assume that $A$ is a matrix with full column rank and row paving $(t,\beta_1, \beta_2)$, and the index $i$ is selected with probability $ 1/t$, it has been demonstrated in \cite[Theorem 1.2]{needell2014paved} that the RBK method \eqref{rbk-gave} satisfies
\begin{equation}
	\label{rbk_convergence}
	\mathbb{E}[\|x^k-x^*\|^2_2]\leq\left(1-\frac{\sigma^2_{\min}(A)}{t\beta_2}\right)^k\|x^0-x^*\|^2_2+\frac{\beta_2}{\beta_1}\frac{\|Ax^*-b\|_2^2}{\sigma^2_{\min}(A)},
\end{equation}
where $x^*$ is  the least-squares solution of the linear system, and $\sigma_{\min}(A)$ denotes the smallest nonzero singular value of $A$. 
This result reveals how the implementation and behavior  of the RBK method are significantly influenced by the properties of the submatrices $A_{\mathcal{S} }$ indexed by the blocks $\mathcal{S}$ in the partition $\Gamma$. Thus, it prompts us to further explore the possibility of designing  alternative block selection strategies for the RBK method, instead of relying on any special row paving.

Recently, Haddock et al. \cite{9869773} introduced the concept of a row covering, which serves as a natural extension of the row paving. Unlike row paving, which requires that the subsets strictly partition the row indices of a matrix, row covering allows these subsets to overlap. Moreover, in the context of row covering, the parameter $\beta_1$ is redefined to provide a lower bound not for the minimum eigenvalue, but specifically for the minimum non-zero eigenvalue of the submatrices formed by these subsets. However, similar to row paving, the convergence of the algorithm still hinges on the parameters \(\beta_1\) and \(\beta_2\) of a good row covering. In this paper, we intend to introduce the volume sampling \cite{deshpande2006matrix} technique to improve the properties of the RBK method.

\subsection{Our contributions}
Our contributions can be summarized as follows.

\begin{itemize}
	\item[1.] 
	We investigate the RBK method with volume sampling (RBKVS) for solving the  linear system \eqref{eq1}. In this approach, a submatrix comprising $s$-subsets of rows from the coefficient matrix $A$ is selected, with the selection probability proportional to the volumes (i.e., determinants) of the principal submatrices of $AA^{\top}$. We analyze the convergence properties of the RBKVS method, demonstrating that the block size $s$ can affect the convergence factor. In particular, if the coefficient matrix $A$ has several relatively large singular values, the RBKVS method may exhibit significant advantages.
	\item[2.] Inspired by the success of the heavy ball momentum technique \cite{polyak1964some,loizou2020momentum,zeng2024adaptive} in improving the performance of optimization methods, we extend this acceleration technique to enhance the performance of the RBKVS method. We demonstrate that the momentum variants of the RBKVS method achieve an accelerated convergence rate.
	\item[3.] We propose an algorithm designed for efficient volume sampling from sparse matrices. The preprocessing complexity of this algorithm is proportional to the number of nonzero elements in the matrix, ensuring efficiency with sparse data.  Specifically, if the coefficient matrix \( A \) of the linear system is fixed and only the right-hand side vector \( b \) varies, our preprocessing step needs to be performed only once. This allows the precomputed results to be reused in subsequent computations, significantly reducing the overall cost.
	\item[4.] Numerical experiments, including comparison with the state-of-the-art methods,  demonstrate that incorporating volume sampling and momentum into the RBK method significantly enhances its accuracy and efficiency. 
\end{itemize}

\subsection{Related work}

There is a vast literature on the Kaczmarz-originated method. Most research in this field typically focuses on developing rules for selecting row indices \cite{censor1981row,karczmarz1937angenaherte,Gow15,bai2018greedy,Nec19,han2022pseudoinverse}, or on obtaining accelerated \cite{liu2016accelerated,loizou2020momentum,su2024greedy,zeng2024adaptive,Han2022-xh}, Bregman \cite{schopfer2019linear,yuan2022adaptively,zeng2023fast,tondji2024acceleration}, nonlinear \cite{gower2024bregman,wang2022nonlinear}, and parallel \cite{liu2014asynchronous,ferreira2024parallelization} variants of existing methods. For a comprehensive overview of the topic, the reader is referred to the recent review papers \cite{bai2023randomized,a2024survey} and the references therein. Below, we discuss several works that are particularly relevant to our research.

As highlighted in the previous section, one of the seminal contributions to randomized block-type Kaczmarz methods is the RBK method, proposed by Needell and Tropp \cite{needell2014paved}.
Another advancement in this area is
the sketch-and-project method developed by Gower and Richt{\'a}rik \cite{Gow15}, which can include the RBK method as a special case. However, it is limited to consistent linear systems, and its convergence upper bound is difficult to estimate.
The randomized averaged block Kaczmarz (RABK) method, introduced by Necoara \cite{Nec19}, represents another significant advancement. Compared to RBK, RABK offers substantial advantages, notably its capability for parallelization and its independence from the requirement for row paving. Recently, the momentum variants of the RABK method have been investigated in \cite{han2022pseudoinverse,zeng2024adaptive}. Inspired by the effectiveness of the heavy ball momentum technique in the stochastic methods, we intend to employ such momentum acceleration in the RBK method.

The most relevant work to ours is \cite{rodomanov2020randomized}, where the authors investigated the convergence rate of block coordinate descent (BCD) algorithms with volume sampling
by exploiting specific formulas involving sums of determinants and adjugate matrices.  Given that the RBK method can be interpreted as a randomized BCD approach (refer to, for instance, \cite{hefny2017rows,zeng2023fast} for further insights), their convergence findings can be extrapolated to the RBK method. However, there are two primary limitations in applying their convergence results to the RBK method. First, their linear convergence conclusions require the assumption that the matrix \(AA^\top\) is positive definite.
Second, they assume that the linear systems are consistent. In contrast, our convergence results are applicable to various types of linear systems, including consistent or inconsistent, full rank or rank-deficient; see Theorem \ref{thm1}. Furthermore, we also explore a variant of the RBK method with momentum, inspired by the success of Polyak's heavy ball momentum method. To the best of our knowledge, this is the first study to investigate the momentum  variants of the RBK method.

Another line of related research is the momentum variants of the stochastic gradient descent (SGD) method \cite{garrigos2023handbook,robbins1951stochastic,loizou2020momentum,han2022pseudoinverse,Han2022-xh}. The RK method, as a specific case of SGD, has been studied within this momentum framework. For example, Loizou and Richt{\'a}rik \cite{loizou2020momentum} explored several types of stochastic optimization algorithms enhanced by heavy ball momentum for solving convex quadratic problems. 
We note that when the block size $s=1$, our convergence results are consistent with those obtained in \cite{loizou2020momentum,Han2022-xh}. Moreover, we have determined that the efficiency of the method consistently improves as the block size increases.

Finally, we note that volume sampling, which can be viewed as a type of determinantal point process (DPP) \cite{macchi1975coincidence}, has recently received significant attention in the literature. It has a wide range of applications in various contexts, including subset selection for matrices \cite{avron2013faster}, randomized numerical linear algebra \cite{derezinski2021determinantal}, and linear regression \cite{derezinski2018reverse}. There has been significant recent work on efficiently performing volume sampling, such as using mixtures of projection DPPs \cite{kulesza2012determinantal}, strongly Rayleigh measures \cite{anari2016monte}, the Markov chain Monte Carlo (MCMC) algorithm \cite{anari2024optimal}, and the randomized Hadamard transform \cite{derezinski2024solving}.

\subsection{Organization}
The remainder of this paper is organized as follows. After introducing some notations and lemmas in Section 2, we analyze the RBKVS method in Section 3. In Section 4, we introduce the heavy ball momentum technique into the RBKVS method and show its improved linear convergence upper bound. In Section 5, we discuss how to efficiently generate a random variable according to volume sampling. In Section 6, we perform some numerical experiments to show the effectiveness of the proposed methods. Finally we conclude the paper in Section 7.

\section{Basic definitions and Preliminaries}

\subsection{Basic definitions}

For any random variables $\xi_1$ and $\xi_2$, we use $\mathbb{E}[\xi_1]$ and $\mathbb{E}[\xi_1\lvert \xi_2]$ to denote the expectation of $\xi_1$ and the conditional expectation of $\xi_1$ given $\xi_2$. For any integer $s$ satisfying $1 \leq  s  \leq  m$, we represent the collection of all $s$-element subsets of $[m]:=\{1,\ldots,m\}$ by $\binom{[m]}{s}$. For a vector $x\in\mathbb{R}^m$, the notations $x_{\mathcal{S}},x^\top$, and $\|x\|_2$ are used to denote the subvector indexed by $\mathcal{S}$, the transpose, and the Euclidean norm of $x$, respectively. We use $\text{diag}(x)$ to denote the diagonal matrix whose entries on the diagonal are the components of $x$. For any $i\in[m]$, we use  $x_{-i}\in \mathbb{R}^{m-1}$ to denote the vector obtained by omitting the $i$-th element from $x$, i.e., $	x_{-i}=(x_1,\ldots,x_{i-1},x_{i+1},\ldots,x_m)^\top$.
For a vector $\xi=(\xi_{1},...,\xi_{m})^\top$, we define its $\ell$-th elementary symmetric polynomial as 
\begin{equation}
	\label{def-esp}
	e_\ell(\xi):=\sum_{1\leq i_1<\cdots < i_\ell\leq m}\xi_{i_1}\cdots  \xi_{i_\ell}, \ 1\leq \ell\leq m,
\end{equation}
and we set $e_0(\xi):=1$ for convenience.

Let  $A$  be an $m\times n$ matrix, we use $A_{\mathcal{S}}$, $A^\top$, $A^\dagger$, $\|A\|_F$, and $\text{Range}(A)$ to denote the row  submatrix indexed by $\mathcal{S} $, the transpose, the Moore-Penrose pseudoinverse, the Frobenius norm, and the column space of $A$, respectively.  
We use $A=U\Sigma V^\top$ to denote the singular value decomposition (SVD) of $A$, where $U\in\mathbb{R}^{m\times m}$, $\Sigma\in\mathbb{R}^{m\times n}$, and $V\in\mathbb{R}^{n\times n}$.
The nonzero singular values of a matrix $A$ are $\sigma_1(A)\geq\sigma_2(A)\geq\ldots\geq\sigma_{r}(A):=\sigma_{\min}(A)>0$, where $r$ is the rank of $A$ and $\sigma_{\min}(A)$ denotes the smallest nonzero singular value of $A$.  We see that $\|A\|_2=\sigma_{1}(A)$ and $\|A\|_F=\sqrt{\sum\limits_{i=1}^r \sigma_i^2(A)}$.  The identity matrix is denoted by $I$. For a square matrix $A$, by $\operatorname{det}(A)$ and $\operatorname{Adj}(A)$ we denote its determinant and adjugate matrix (the transpose of the cofactor matrix), respectively.

A symmetric matrix $M\in\mathbb{R}^{n\times n}$ is said to be positive semidefinite if $x^\top M x\geq 0$ holds for any
$x\in\mathbb{R}^n$.  For any
two matrices $M$ and $N$, we write $M\preceq N$  to represent $N-M$ is positive semidefinite.

\begin{definition}[Volume sampling] 
	\label{vs}
	Let \( A \) be an \( m \times n \) matrix and suppose the integer \( s \) satisfies \( 1 \leq s \leq \operatorname{rank}(A) \). Consider a random variable \( \mathcal{S}_0 \) that takes values in \( \binom{[m]}{s} \). We define \( \mathcal{S}_0 \) as being generated according to \( s \)-element volume sampling with respect to \( AA^\top \), denoted by \( \mathcal{S}_0 \sim \operatorname{Vol}_s(AA^\top) \), if for all \( \mathcal{S} \in \binom{[m]}{s} \), the probability that \( \mathcal{S}_0 \) equals \( \mathcal{S} \) is given by
	\[
	\mathbb{P}(\mathcal{S}_0 = \mathcal{S}) = \frac{\operatorname{det}(A_{\mathcal{S}} A_{\mathcal{S}}^\top)}{\sum_{\mathcal{J} \in \binom{[m]}{s}} \operatorname{det}(A_{\mathcal{J}} A_{\mathcal{J}}^\top)}.
	\]
\end{definition}

In this paper, we denote \(x^*\) as a specific solution of the linear system \eqref{eq1}. For any \(x^0 \in \mathbb{R}^n\), we define
$ x^0_* := A^\dagger b + (I - A^\dagger A)x^0$.
We note that \(x^0_*\) is the orthogonal projection of \(x^0\) onto the set
$ \{x \in \mathbb{R}^n \mid A^\top A x = A^\top b\}$.

\subsection{Some useful lemmas}

The lemmas presented below are integral to our convergence analysis, and their detailed proofs can be found in the appendix.

\begin{lemma}\label{keylemma-exp1}
	Suppose \( A \) is an \( m \times n \) matrix with the singular value decomposition given by \( A = U \Sigma V^\top \). Define \( \lambda := (\sigma^2_1(A), \ldots,\sigma^2_{\operatorname{rank}(A)}(A),0,\ldots,0 )^\top \in\mathbb{R}^m\), which represents the eigenvalues of $AA^\top$.
	Assume \(1 \leq s \leq \operatorname{rank}(A)\), and let the random variable \(\mathcal{S}\) be distributed according to \(\mathcal{S} \sim \operatorname{Vol}_s(AA^\top)\). Then,
	\[
	\mathbb{E}[A_{\mathcal{S} }^\dagger A_{\mathcal{S} }] = A^\top H_s A,
	\]
	where 
	\begin{equation}
		\label{def_HB}
		H_s := \frac{U \operatorname{diag}\left(e_{s-1}(\lambda_{-1}), \ldots, e_{s-1}(\lambda_{-m})\right) U^\top}{e_{s}(\lambda)}
	\end{equation}
	and \( \lambda_{-i} \in \mathbb{R}^{m-1} \) denotes the vector \( \lambda \) with the \( i \)-th element omitted.
\end{lemma}

\begin{lemma}\label{keylemma-exp2}
	Under the notations of  Lemma \ref{keylemma-exp1}. It holds that
	\[
	\mathbb{E}[I^\top_{\mathcal{S} } (A_{\mathcal{S} }^{\dagger})^\top A_{\mathcal{S} }^{\dagger}  I_{\mathcal{S} }] \preceq   H_s.
	\]
\end{lemma}

The next lemma  will be utilized to establish the convergence factor for the algorithms.
\begin{lemma}
	\label{lemma-lowerbound}
	Under the notations of Lemma \ref{keylemma-exp1}, for any \( z \in \operatorname{Range}(A^\top) \), we have
	\[
	z^\top A^\top H_s A z \geq \frac{\sigma_{\min}^2(A)}{\sigma_s^2(A) + \cdots + \sigma_{\min}^2(A)} \|z\|_2^2.
	\]
\end{lemma}

The following result can be derived from the proof of Lemma 3.5 in \cite{rodomanov2020randomized}. Here, we include a simple proof for
completeness.

\begin{lemma}\label{lemma-xie-1125}
	Under the notations of Lemma \ref{keylemma-exp1}, for any $i\in[m]$, we have
	$$\frac{e_{s-1}(\lambda_{-i})}{e_s(\lambda)}\geq \frac{1}{\sigma^2_{\min\{i,s\}}(A)+\sum_{j=s+1}^{\operatorname{rank}(A)}\sigma^2_j(A)}.$$
\end{lemma}
\begin{proof}
	For any fixed $i\in[m]$, we have $$e_s(\lambda)=e_{s}(\lambda_{-i})+\lambda_i e_{s-1}(\lambda_{-i})\leq e_{s-1}(\lambda_{-i})\sum_{j=s}^{m}(\lambda_{-i})_{j}+\lambda_i e_{s-1}(\lambda_{-i}).$$
	Hence, 
	$$\frac{e_{s-1}(\lambda_{-i})}{e_s(\lambda)}\geq \frac{1}{\lambda_i+\sum_{j=s}^{m}(\lambda_{-i})_{j}}=\frac{1}{\lambda_{\min\{i,s\}}+\sum_{j=s+1}^{m}\lambda_{j}}= \frac{1}{\sigma^2_{\min\{i,s\}}(A)+\sum_{j=s+1}^{\operatorname{rank}(A)}\sigma^2_j(A)},$$
	where the last equality follows from the fact that $\lambda_j=\sigma_j^2(A)$ for $1\leq j\leq \operatorname{rank}(A)$ and  $\lambda_{\operatorname{rank}(A)+1}=\cdots=\lambda_m=0$.
\end{proof}

\section{Randomized block Kaczmarz with volume sampling}

In this section, we introduce the randomized block Kaczmarz method with volume sampling (RBKVS) for solving linear systems and analyze its convergence properties. The method is formally described in Algorithm \ref{RBKVS}. It can be observed that the canonical RK method can be recovered when the block size \(s=1\), with the index \(i_k\) chosen with probability \(\mathbb{P}(i_k=i)=\frac{\|A_i\|^2_2}{\|A\|^2_F}\).

\begin{algorithm}
\caption{Randomized block Kaczmarz with volume sampling (RBKVS)}
\label{RBKVS}
\begin{algorithmic}
	\Require $A\in\mathbb{R}^{m\times n},b\in\mathbb{R}^m$, $k=0$, initial point $x^0\in \mathbb{R}^n$, and $1\leq s \leq\text{rank}(A)$.
	\begin{enumerate}
		\item[1:]  Select $\mathcal{S}_k\sim\mbox{Vol}_{s}(AA^\top)$.
		\item[2:] Update $x^{k+1}=x^k-A_{\mathcal{S}_k }^{\dagger}(A_{\mathcal{S}_k }x^k-b_{\mathcal{S}_k})$. 
		\item[3:] If the stopping rule is satisfied, stop and go to output. Otherwise, set $k=k+1$ and return to Step $1$.
	\end{enumerate}
	\Ensure
	The approximate solution.
\end{algorithmic}
\end{algorithm}

\subsection{The linear convergence of the expected norm of the error}

In this subsection, we study the convergence properties of Algorithm \ref{RBKVS}. For convenience, given any matrix \( A \in \mathbb{R}^{m \times n} \) and \( 1 \leq s \leq \operatorname{rank}(A) \), we define
\begin{equation}
	\label{xie-def-rho}
	\rho_{A,s}:=\frac{\sigma_{\min}^2(A)}{ \sigma^2_{s}(A)+\cdots+ \sigma_{\min}^2(A) }.
\end{equation}

We have the following convergence result for Algorithm \ref{RBKVS} in solving different types of linear systems, whether consistent or inconsistent, full rank or rank-deficient.

\begin{theorem}\label{thm1}
	
	For any given linear system \(Ax = b\), let \(\{x^k\}_{k \geq 0}\) be the sequence generated by Algorithm \ref{RBKVS} with an arbitrary initial vector \(x^0 \in \mathbb{R}^n\). Define \(x^0_* = A^\dagger b + (I - A^\dagger A)x^0\) and \(r^* = Ax^0_* - b\).  Then 
	$$
		\mathbb{E}[\lVert x^{k}-x^0_*\rVert^2_2] 
		\leq \left(1- \rho_{A,s}\right)^k\lVert x^0-x^0_*\rVert^2_2+\frac{e_{s-1}(\lambda_{-m})}{ e_{s}(\lambda) \rho_{A,s} }\lVert r^*\rVert^2_2,
		$$
	where $\rho_{A,s}$ is given by \eqref{xie-def-rho} and $\lambda:=(\sigma_1^2(A),\ldots,\sigma_{\min}^2(A),0,\ldots,0)^\top\in\mathbb{R}^m$.
\end{theorem}

\begin{proof}
	According to the Step 2 in Algorithm \ref{RBKVS} and $b=Ax^0_*-r^*$, we have
	$$x^{k+1}-x^0_*=(I-A_{\mathcal{S}_k }^{\dagger}A_{\mathcal{S}_k })(x^k-x^0_*)-A_{\mathcal{S}_k }^{\dagger}r^*_{\mathcal{S}_k}.$$
	Since the range spaces of $A_{\mathcal{S}_k}^{\dagger}$ and $I-A_{\mathcal{S}_k}^{\dagger}A_{\mathcal{S}_k}$ are orthogonal, the Pythagorean Theorem implies
	$$
	\begin{aligned}
		\lVert x^{k+1}-x^0_*\rVert^2_2
		&=\lVert (I-A_{\mathcal{S}_k }^{\dagger}A_{\mathcal{S}_k })(x^k-x^0_*)\rVert^2_2+\lVert A_{\mathcal{S}_k }^{\dagger}r^*_{\mathcal{S}_k}\rVert^2_2\\
		&= (x^k-x^0_*)^\top (I-A_{\mathcal{S}_k }^{\dagger}A_{\mathcal{S}_k })(x^k-x^0_*)+ (r^*)^\top I_{\mathcal{S}_k }^\top  (A_{\mathcal{S}_k }^{\dagger})^\top A_{\mathcal{S}_k }^{\dagger} I_{\mathcal{S}_k } r^*.
	\end{aligned}
	$$
	Thus we have
	\begin{equation}
		\label{proof-xie-11242}
		\begin{aligned}
			\mathbb{E} [\lVert x^{k+1}-x^0_*\rVert^2_2 \mid x^k]
			= &(x^k-x^0_*)^\top (I-\mathbb{E} [A_{\mathcal{S}_k }^{\dagger}A_{\mathcal{S}_k }])(x^k-x^0_*) \\
			&+(r^*)^\top \mathbb{E} [I_{\mathcal{S}_k }^\top  (A_{\mathcal{S}_k }^{\dagger})^\top A_{\mathcal{S}_k }^{\dagger} I_{\mathcal{S}_k }]r^* \\
			\leq &\lVert x^k-x^0_*\rVert^2_2-(x^k-x^0_*)^\top  A^\top H_s A(x^k-x^0_*)+(r^*)^\top H_s r^*,
		\end{aligned}
	\end{equation}
	where the last inequality follows from Lemmas \ref{keylemma-exp1} and \ref{keylemma-exp2}. By the definition of $H_s$ as given in \eqref{def_HB}, we know that $\frac{e_{s-1}(\lambda_{-m})}{e_{s}(\lambda)}I\succeq H_s$. Hence,
	\begin{equation}
		\label{proof-xie-11243}
		(r^*)^\top H_s r^*\leq \frac{e_{s-1}(\lambda_{-m})}{e_{s}(\lambda)} \|r^*\|^2_2.
	\end{equation}
	By the iteration scheme of Algorithm \ref{RBKVS} and noting that $\operatorname{Range}(A^\dagger_{\mathcal{S}_k})=\operatorname{Range}(A^\top_{\mathcal{S}_k })\subseteq
	\operatorname{Range}(A^\top)$, we know that $x^k\in x^0+\operatorname{Range}(A^\top)$. Hence, 
	$$x^k-x^0_*\in x^0-x^0_*+\operatorname{Range}(A^\top)=A^\dagger (Ax^0-b)+\operatorname{Range}(A^\top)=\operatorname{Range}(A^\top),$$
	which together  with Lemma \ref{lemma-lowerbound} implies that
	\begin{equation}\label{eqax}
		(x^k-x^0_*)^\top  A^\top H_s A(x^k-x^0_*)\geq \rho_{A,s}\lVert x^k-x^0_*\rVert^2_2.
	\end{equation}
	Substituting \eqref{proof-xie-11243} and \eqref{eqax} into \eqref{proof-xie-11242},  we have
	\begin{equation*}
		\mathbb{E} [\lVert x^{k+1}-x^0_*\rVert^2_2\mid x^k]=\left(1-\rho_{A,s}\right)\lVert x^k-x^0_*\rVert^2_2+\frac{ e_{s-1}(\lambda_{-m})}{ e_{s}(\lambda)}\lVert r^*\rVert^2_2.
	\end{equation*}
	Taking the expectation over the entire history, we have
	\begin{equation*}
		\begin{aligned}
			\mathbb{E} [\lVert x^k-x^0_*\rVert^2_2]
			&\leq\left(1-\rho_{A,s} \right)^k\lVert x^0-x^0_*\rVert^2_2+\frac{e_{s-1}(\lambda_{-m})}{e_{s}(\lambda)}\lVert r^*\rVert^2_2\left(\sum_{i=0}^{k-1}\left(1-\rho_{A,s}\right)^i\right)\\
			&\leq\left(1-\rho_{A,s}\right)^k\lVert x^0-x^0_*\rVert^2_2+\frac{e_{s-1}(\lambda_{-m})}{e_{s}(\lambda)\rho_{A,s}}\lVert r^*\rVert^2_2
		\end{aligned}
	\end{equation*}
	as desired. This completes the proof of this theorem.
\end{proof}

If the block size $s=1$, then Algorithm \ref{RBKVS} recovers the canonical RK method, and Theorem  \ref{thm1} now becomes
$$
\mathbb{E}[\lVert x^{k}-x^0_*\rVert^2_2] 
\leq \left(1- \frac{\sigma^2_{\min}(A)}{\|A\|^2_F}\right)^k\lVert x^0-x^0_*\rVert^2_2+\frac{\lVert r^*\rVert^2_2}{\sigma^2_{\min}(A) },
$$
which coincides with the convergence result of RK \cite{strohmer2009randomized,needell2010randomized,lok2024subspace}.
Assume that the linear system \(Ax = b\) is consistent (\(r^* = 0\)) and note that \(1 - \rho_{A,s} \leq e^{-\rho_{A,s}}\), then Algorithm \ref{RBKVS} requires 
\[
K_{s}:= \frac{1}{\rho_{A,s}} \log\left(\frac{\lVert x^0 - x^0_* \rVert^2_2}{\varepsilon}\right).
\]
number of iterations  to achieve accuracy \(\varepsilon\) in terms of the expected norm of the error.

Let us consider \(1 \leq s_1 < s_2 \leq \operatorname{rank}(A)\) and compare the efficiency estimates of Algorithm \ref{RBKVS} for solving a consistent linear system with \(s_1\) coordinates to that with \(s_2\) coordinates. We obtain that
\begin{equation}\label{rate-improved}
	\frac{K_{s_1}}{K_{s_2}}=\frac{\rho_{A,s_2}}{\rho_{A,s_1}}=\frac{\sum_{i=s_1}^{\operatorname{rank}(A)} \sigma^2_i(A)}{\sum_{i=s_2}^{\operatorname{rank}(A)} \sigma^2_i(A)}.
\end{equation}
This implies that increasing the block size from \(s_1\) to \(s_2\) can enhance the iteration complexity of Algorithm \ref{RBKVS} by a factor of \(\frac{\sum_{i=s_1}^{\operatorname{rank}(A)} \sigma^2_i(A)}{\sum_{i=s_2}^{\operatorname{rank}(A)} \sigma^2_i(A)}\). Notably, when \(A\) has several large singular values, this ratio can be substantial.

In particular, we compare Algorithm \ref{RBKVS} with the RK method in terms of computational cost. Assuming that selecting \(\mathcal{S}_k\) is a low-cost operation, which we will discuss how to do efficiently in Section \ref{vs-eff}, this implies that the primary computational cost in each step of Algorithm \ref{RBKVS} is Step 2. Hence, the computational cost per step for Algorithm \ref{RBKVS} and the RK method is \(O(s^2n)\) and \(O(n)\), respectively. Consequently, the total computational cost for Algorithm \ref{RBKVS} and the RK method to achieve accuracy \(\varepsilon\) is
$$
O\left(s^2n K_{s}\right) \ \text{and} \ O\left(n K_{1}\right),
$$
respectively. Thus, if \(s^2\sum_{i=s}^{\operatorname{rank}(A)} \sigma^2_i(A)\leq \sum_{i=1}^{\operatorname{rank}(A)} \sigma^2_i(A)=\|A\|_F^2\), meaning $A$ has several relatively large singular values, then theoretically, Algorithm \ref{RBKVS} is superior to the RK method. 

To better understand this, let us consider a commonly encountered scenario.  
Consider the matrix
$$
A=\left[e_1+\delta e_2,\ldots,e_1+\delta e_{n+1} \right]^\top\in\mathbb{R}^{(n+1)\times n},
$$
where $\delta>0$ and $e_1,\ldots,e_{n+1}$ are the standard basis vectors in $\mathbb{R}^{n+1}$. This matrix is often used to illustrate the error bound in the subset selection problem; see \cite{boutsidis2014near,cai2024interlacing}. 
We have
\[
\sigma_1(A) = \sqrt{n+\delta^2}, \quad \sigma_2(A) = \cdots = \sigma_n(A) = \delta.
\]
Assume that \(s \geq 2\). Now,
$
s^2\sum_{i=s}^{\operatorname{rank}(A)} \sigma^2_i(A) \leq \sum_{i=1}^{\operatorname{rank}(A)} \sigma^2_i(A)
$
is equivalent to
\[
s^2(n-s+1)\delta^2 \leq n + n\delta^2.
\]
Hence, if \(\delta\) is small enough, e.g., \(\delta = \frac{1}{n}\), the above inequality holds. In such a case, we know that the first singular value of \(A\) is relatively large, while the subsequent singular values are smaller.

\subsection{The linear convergence of the norm of the expected error}

In the next section, we will compare the RBKVS method and its momentum variant (mRBKVS), demonstrating that mRBKVS achieves accelerated linear convergence. 
To this end, we present results on the convergence of the norm of the expected error \(\lVert \mathbb{E}[x^k - x_*^0]\rVert_2^2\) for RBKVS.

\begin{theorem} 
	\label{Theroem 2}
	Suppose that the linear system \eqref{eq1} is consistent and $x^0\in\mathbb{R}^n$ is an arbitrary initial vector. Let $x_*^0=A^\dagger b+(I-A^\dagger A)x^0$. Then the iteration sequence $\{x^k\}_{k\geq 0}$ generated by Algorithm {\rm\ref{RBKVS}}  satisfies
	$$
	\lVert \mathbb{E}[x^{k}-x_*^0]\rVert^2_2 \leq \left(1- \rho_{A,s}\right)^{2k} \lVert x^0-x_*^0\rVert^2_2,
	$$
	where $\rho_{A,s}$ is given by \eqref{xie-def-rho}.
\end{theorem}

\begin{proof}
	From the Step 2 of Algorithm \ref{RBKVS}, we know that
	$$
	x^{k+1}-x_*^0=(I- A_{\mathcal{S}_k}^{\dagger}A_{\mathcal{S}_k})(x^k-x_*^0).
	$$
	Taking expectations and using Lemma \ref{keylemma-exp1}, we have
	$ \mathbb{E}[x^{k+1}-x_*^0 ]
	=(I- A^{\top}H_s A) \mathbb{E}[x^k-x_*^0]$
	and hence,
	$$
	\mathbb{E}[x^{k}-x_*^0 ]=(I- A^{\top}H_s A)^k (x^0-x_*^0).
	$$
	Then, we have
	$$
	\|\mathbb{E}[x^{k}-x_*^0 ]\|^2_2=(x^0-x_*^0)^\top (I- A^{\top}H_s A)^{2k} (x^0-x_*^0)\leq \left(1- \rho_{A,s}\right)^{2k} \lVert x^0-x_*^0\rVert^2_2,
	$$
	where the last inequality follows from $x^0-x^0_*=A^\dagger(Ax^0-b)\in\operatorname{Range}(A^\top)$ and Lemma \ref{lemma-lowerbound}. 
	This completes the proof.
\end{proof}

The similarity between \(\mathbb{E} \left[\lVert x^k-x_*^0 \rVert^2_2\right]\) and \(\lVert \mathbb{E}[x^k-x_*^0] \rVert^2_2\) might lead to confusion. The convergence of \(\lVert \mathbb{E}[x^k-x_*^0] \rVert^2_2\) is actually weaker than that of \(\mathbb{E} \left[\lVert x^k-x_*^0 \rVert^2_2\right]\). By definition, if \(\mathbb{E}[x^k]\) is bounded for all \(x^k \in \mathbb{R}^n\), we have
\[
\mathbb{E}[\lVert x^k-x_*^0 \rVert^2_2] = \lVert \mathbb{E}[x^k-x_*^0] \rVert^2_2 + \mathbb{E}\left[\lVert x^k-\mathbb{E}[x^k] \rVert^2_2\right],
\]
which implies that the convergence of \(\mathbb{E} \left[\lVert x^k-x_*^0 \rVert^2_2\right]\) ensures the convergence of \(\lVert \mathbb{E}[x^k-x_*^0] \rVert^2_2\), but not vice versa.
In the subsequent section, we observe that although the expected norm of the error \(\mathbb{E} \left[\lVert x^k-x^0_* \rVert^2_2\right]\) for mRBKVS also converges linearly, its rate of convergence is slower than that of RBKVS. Consequently, we also examine the norm of the expected error \(\lVert \mathbb{E}[x^k-x^0_*] \rVert^2_2\), where mRBKVS demonstrates an advantage over RBKVS. This is why we also investigate the convergence of \(\lVert \mathbb{E}[x^k-x^0_*] \rVert^2_2\) for RBKVS.

\section{Acceleration by heavy ball momentum}
In this section, we provide the momentum-induced RBKVS (mRBKVS) method for solving the  linear system \eqref{eq1}. First, we give a short description of the heavy ball momentum (HBM) method. Consider the  unconstrained minimization problem
$
\min\limits_{x\in\mathbb{R}^n} f(x),
$
where $f$ is a differentiable convex function.
To solve this problem, the gradient descent method
with momentum of Polyak \cite{polyak1964some} takes the form
\begin{equation}\label{hbm}
	x^{k+1}=x^{k}-\omega \nabla f\big(x^{k}\big)+\beta\big(x^{k}-x^{k-1}\big),
\end{equation}
where $\omega>0$ is the stepsize, $\beta$ is the momentum parameter, and $\nabla f\left(x^{k}\right)$ denotes the gradient of $f$ at $x^k$.  When $\beta=0$, the method reduces to the standard gradient descent method. If the full gradient in \eqref{hbm} is replaced by the unbiased estimate of the gradient, then it becomes the  stochastic
HBM (SHBM) method.
In \cite{ghadimi2015global}, the authors showed that the deterministic HBM method
converges globally and sublinearly for smooth and convex functions. For the SHBM, one may refer to \cite{sebbouh2021almost,garrigos2023handbook} for more discussions.

Inspired by the success of the SHBM method, in this section  we incorporate the HBM into our RBKVS method, obtaining the mRBKVS method described in Algorithm \ref{RBKVSm}.  In the rest of this section, we will study the convergence properties of the proposed mRBKVS method.

\begin{algorithm}
\caption{RBKVS with momentum (mRBKVS)}
\label{RBKVSm}
\begin{algorithmic}
\Require $A\in\mathbb{R}^{m\times n},b\in\mathbb{R}^m$, $k=0$, initial point $x^0=x^1\in\mathbb{R}^n$, and $1\leq s\leq\text{rank}(A)$.
\begin{enumerate}
\item[1:]  Select $\mathcal{S}_k\sim\mbox{Vol}_{s}(AA^\top)$.

\item[2:] Compute $x^{k+1}=x^k-\omega A_{\mathcal{S}_k}^{\dagger}(A_{\mathcal{S}_k}x^k-b_{\mathcal{S}_k})+\beta(x^k-x^{k-1})$.

\item[3:] If the stopping rule is satisfied, stop and go to output. Otherwise, set $k=k+1$ and return Step 2.
\end{enumerate}
\Ensure
The approximate solution.
\end{algorithmic}
\end{algorithm}

\subsection{Convergence of iterates}
We have the following convergence result for Algorithm \ref{RBKVSm}.

\begin{theorem}
	\label{thm4}
	Suppose that the linear system \eqref{eq1} is consistent and initial points $x^0=x^1$. Let $x^0_*=A^\dagger b+(I-A^\dagger A)x^0$, $\rho_{A,s}$ be given by \eqref{def_HB}, and $\lambda=(\sigma_1^2(A),\ldots,\sigma_{\min}^2(A),0,\ldots,0)^\top\in\mathbb{R}^m$. Assume $0<\omega<2$ and $\beta\geq0$, and that the expressions 
	$$\gamma_1:=1+3\beta+2\beta^2-\omega(2-\omega+\beta)\rho_{A,s} \ \ \text{and} \ \ \gamma_2:=\beta+2\beta^2+\omega\beta\frac{e_{s-1}(\lambda_{-m})}{e_{s}(\lambda)}\|A\|^2_2$$
	satisfy $\gamma_1+\gamma_2<1$. Then the iteration sequence $\{x^k\}_{k\geq 0}$ generated by Algorithm \ref{RBKVSm} satisfies 
	\begin{equation}
		\mathbb{E}[\lVert x^{k}-x^0_*\rVert^2_2 ]
		\leq \rho^k(1+q)\lVert x^0-x^0_*\rVert^2_2,
	\end{equation}
	where $\rho=\frac{\gamma_1+\sqrt{\gamma_1^2+4\gamma_2}}{2}$ and $q=\rho-\gamma_1$.
\end{theorem}

The following  lemma  is essential to our proof.
\begin{lemma}[\cite{loizou2020momentum}, Lemma 8.1]	
	\label{lemma2}	
	Fix $F_1=F_0\geq 0$, and let $\{F_k\}_{k\geq 0}$ be a sequence of nonnegative real numbers satisfying the relation 
	$$F_{k+1}\leq \gamma_1 F_k +\gamma_2F_{k-1}, \ \forall \ k\geq 1,$$
	where $\gamma_2\geq 0, \gamma_1+\gamma_2<1$ and at least one of the coefficients $\gamma_1,\gamma_2$ is positive. Then the sequence satisfies the relation 
	$$F_{k+1}\leq \rho^k(1+q)F_0, \ \forall \ k\geq 1,$$
	where $\rho=\frac{\gamma_1+\sqrt{\gamma_1^2+4\gamma_2}}{2}$ and $q=\rho-\gamma_1\geq 0$.
\end{lemma}

\begin{proof}[Proof of Theorem \ref{thm4}]
	First, we have
	\begin{equation*}
		\begin{aligned}
			\lVert x^{k+1}-x^0_*\rVert^2_2
			&=\lVert x^k-\omega A_{\mathcal{S}_k}^{\dagger}A_{\mathcal{S}_k}(x^k-x^0_*)+\beta(x^k-x^{k-1})-x^0_*\rVert^2_2\\
			&=\underbrace{\lVert x^k-x^0_*-\omega A_{\mathcal{S}_k}^{\dagger}A_{\mathcal{S}_k}(x^k-x^0_*)\rVert^2_2}_{T_1}
			+ \underbrace{\beta^2\lVert x^k-x^{k-1}\rVert^2_2}_{T_2}\\
			&\ \ \ +\underbrace{2\beta\left\langle x^k-x^0_*-\omega A_{\mathcal{S}_k}^{\dagger}A_{\mathcal{S}_k}(x^k-x^0_*),x^k-x^{k-1}\right\rangle}_{T_3}.
		\end{aligned}
	\end{equation*}
	We now analyze the three expressions $T_1,T_2$ and $T_3$ separately. The first term can be written as
	\begin{equation*}
		\begin{aligned}
			T_1
			&=(x^k-x^0_*)^\top(I-2\omega A_{\mathcal{S}_k}^{\dagger}A_{\mathcal{S}_k}+\omega^2 A_{\mathcal{S}_k}^{\dagger}A_{\mathcal{S}_k})(x^k-x^0_*)\\
			&=\lVert x^k-x^0_*\rVert^2_2-\omega(2-\omega)(x^k-x^0_*)^\top A_{\mathcal{S}_k}^{\dagger}A_{\mathcal{S}_k}(x^k-x^0_*).
		\end{aligned}
	\end{equation*}
	The second term can be bounded as
	\begin{equation*}
		\begin{aligned}
			T_2=\beta^2\lVert (x^k-x^0_*)+(x^0_*-x^{k-1})\rVert^2_2\leq 2\beta^2\lVert x^k-x^0_*\rVert^2_2+2\beta^2\lVert x^{k-1}-x^0_*\rVert^2_2.
		\end{aligned}
	\end{equation*}
	We now bound the third term
	\begin{equation*}
		\begin{aligned}
			T_3
			&=2\beta \left\langle x^k-x^0_*,x^k-x^0_*+x^0_*-x^{k-1}\right\rangle
			+2\omega\beta \left\langle A_{\mathcal{S}_k}^{\dagger}A_{\mathcal{S}_k}(x^k-x^0_*),x^{k-1}-x^k\right\rangle\\
			&=2\beta\lVert x^k-x^0_*\rVert^2_2
			+2\beta\langle x^k-x^0_*,x^0_*-x^{k-1}\rangle
			+2\omega\beta \left\langle A_{\mathcal{S}_k}^{\dagger}A_{\mathcal{S}_k}(x^k-x^0_*),x^{k-1}-x^k\right\rangle\\
			&=\beta\lVert x^k-x^0_*\rVert^2_2
			+\beta\lVert x^k-x^{k-1}\rVert^2_2-\beta\lVert x^{k-1}-x^0_{*}\rVert^2_2
			+2\omega\beta \left\langle A_{\mathcal{S}_k}^{\dagger}A_{\mathcal{S}_k}(x^k-x^0_*),x^{k-1}-x^k\right\rangle\\
			&\leq 3\beta\lVert x^k-x^0_*\rVert^2_2
			+\beta\lVert x^{k-1}-x^0_{*}\rVert^2_2
			+2\omega\beta \left\langle A_{\mathcal{S}_k}^{\dagger}A_{\mathcal{S}_k}(x^k-x^0_*),x^{k-1}-x^k\right\rangle.
		\end{aligned}
	\end{equation*}
	Combining the above inequalities, we have 
	\begin{equation*}
		\begin{aligned}
			\lVert x^{k+1}-x^0_*\rVert^2_2
			&\leq
			(1+3\beta+2\beta^2)\lVert x^k-x^0_*\rVert^2_2-\omega(2-\omega)(x^k-x^0_*)^\top A_{\mathcal{S}_k}^{\dagger}A_{\mathcal{S}_k}(x^k-x^0_*)\\
			&\ \ \ +(\beta+2\beta^2)\lVert x^{k-1}-x^0_*\rVert^2_2
			+2\omega\beta\langle A_{\mathcal{S}_k}^{\dagger}A_{\mathcal{S}_k}(x^k-x^0_*),x^{k-1}-x^k\rangle\\
			&\leq
			(1+3\beta+2\beta^2)\lVert x^k-x^0_*\rVert^2_2-\omega(2-\omega+\beta)(x^k-x^0_*)^\top A_{\mathcal{S}_k}^{\dagger}A_{\mathcal{S}_k}(x^k-x^0_*)\\
			&\ \ \ +(\beta+2\beta^2)\lVert x^{k-1}-x^0_*\rVert^2_2
			+\omega\beta\lVert A_{\mathcal{S}_k}^{\dagger}A_{\mathcal{S}_k}(x^{k-1}-x^0_*)\rVert^2_2,
		\end{aligned}
	\end{equation*}
	where the last inequality follows from the convexity of $f_{\mathcal{S}_k}(x):=\frac{1}{2}\lVert A_{\mathcal{S}_k}^{\dagger}A_{\mathcal{S}_k}(x-x^0_*)\rVert^2_2$.
	By taking expectation, we have
	\begin{equation*}
		\begin{aligned}
			\mathbb{E}[\lVert x^{k+1}-x^0_*\rVert^2_2\mid x^k]
			\leq
			&(1+3\beta+2\beta^2)\lVert x^k-x^0_*\rVert^2_2
			+(\beta+2\beta^2)\lVert x^{k-1}-x^0_*\rVert^2_2\\
			&
			-\omega(2-\omega+\beta)(x^k-x^0_*)^\top \mathbb{E}[ A_{\mathcal{S}_k}^{\dagger}A_{\mathcal{S}_k}](x^k-x^0_*)\\
			& +\omega\beta(x^{k-1}-x^0_*)^\top \mathbb{E}[ A_{\mathcal{S}_k}^{\dagger}A_{\mathcal{S}_k}](x^{k-1}-x^0_*)\\
			=&(1+3\beta+2\beta^2)\lVert x^k-x^0_*\rVert^2_2
			+(\beta+2\beta^2)\lVert x^{k-1}-x^0_*\rVert^2_2\\
			&
			-\omega(2-\omega+\beta)(x^k-x^0_*)^\top A^\top H_s A(x^k-x^0_*)\\
			&+\omega\beta(x^{k-1}-x^0_*)^\top A^\top H_s A (x^{k-1}-x^0_*)\\
			\leq&
			\underbrace{\left(1+3\beta+2\beta^2-\omega(2-\omega+\beta)\rho_{A,s}\right)}_{\gamma_1}\lVert x^k-x^0_*\rVert^2_2 \\
			&+\underbrace{\left(\beta+2\beta^2+\omega\beta\frac{e_{s-1}(\lambda_{-m})}{e_{s}(\lambda)}\|A\|^2_2\right)}_{\gamma_2}\lVert x^{k-1}-x^0_*\rVert^2_2,      
		\end{aligned}
	\end{equation*}
	where the first equality follows from Lemma \ref{keylemma-exp1} and  the last inequality follows from  $x^0-x^0_*\in\operatorname{Range}(A^\top)$ and Lemma \ref{lemma-lowerbound}, and  $\frac{e_{s-1}(\lambda_{-m})}{e_{s}(\lambda)}I\succeq H_s$.
	By taking expectation again and letting $F_{k}:=\mathbb{E}[\lVert x^k-x^0_*\rVert^2_2]$, we obtain
	$$F_{k+1}\leq \gamma_1 F_k +\gamma_2F_{k-1}.$$
	The conditions of the Lemma \ref{lemma2} are satisfied. Hence by applying Lemma \ref{lemma2}, we complete the proof.
\end{proof}

\begin{remark}
	Let us compare the convergence rate obtained in Theorem \ref{thm1} and Theorem \ref{thm4}. From the definition of $\gamma_1$ and $\gamma_2$, we know that the convergence rate $q$ in Theorem \ref{thm4} can be viewed as a function of $\beta$. Since $\beta \geq 0$ and $\gamma_1+\gamma_2<1$, it implies that $\gamma_1 \gamma_2+\gamma_2^2=\gamma_2(\gamma_1+\gamma_2)\leq \gamma_2$. Therefore, $\gamma_1^2+4\gamma_2\geq(\gamma_1+2\gamma_2)^2$, and thus
	\begin{equation*}
		\begin{aligned}
			\rho(\beta)
			&\geq \gamma_1+\gamma_2 
			= 1+\left(4\beta+4\beta^2+\beta\omega\left(\frac{e_{s-1}(\lambda_{-m})}{e_{s}(\lambda)}\|A\|^2_2-\rho_{A,s}\right)\right)-\omega(2-\omega)\rho_{A,s}\\
			&\geq 1-\omega(2-\omega)\rho_{A,s} = \rho(0).
		\end{aligned}
	\end{equation*}
	Clearly the lower bound on $\rho$ is an increasing function of $\beta$, which implies that when $\omega=1$, for any $\beta$ the rate is inferior to that of Algorithm \ref{RBKVS} in Theorem \ref{thm1}.
\end{remark}

\subsection{Accelerated linear rate for expected iterates}

In this subsection, we examine the convergence property of the norm of the expected error \(\lVert \mathbb{E}[x^k - x^0_*]\rVert^2_2\). We aim to demonstrate that with an appropriate choice of the relaxation parameter \(\alpha\) and the momentum parameter \(\beta\), Algorithm \ref{RBKVSm} achieves an accelerated linear convergence rate.

\begin{theorem}
	\label{thm5}
	Suppose that the linear system \eqref{eq1} is consistent and $x^0=x^1$ are arbitrary initial vectors. Let $x^0_*=A^\dagger b+(I-A^\dagger A)x^0$,  $\lambda=(\sigma_1^2(A),\ldots,\sigma_{\min}^2(A),0,\ldots,0)^\top\in\mathbb{R}^m$, and $\rho_{A,s}$ be given by \eqref{def_HB}.
	Assume $0<\omega\leq \frac{e_{s}(\lambda)}{e_{s-1}(\lambda_{-m})\|A\|^2_2}$ and $ \left(1-\sqrt{\omega\rho_{A,s}}\right)^2 < \beta < 1$.
	Let $\{x^k\}_{k\geq 0}$ be the iteration sequence generated by Algorithm \ref{RBKVSm}. Then there exists a constant $c>0$ such that 
	\begin{equation*}
		\lVert \mathbb{E}[x^k-x_*^0]\rVert^2_2\leq \beta^k c.
	\end{equation*}
\end{theorem}

\begin{remark}
	Note that the convergence factor specified in Theorem \ref{thm5} exactly matches the value of the momentum parameter \(\beta\). According to Theorem \ref{Theroem 2}, Algorithm \ref{RBKVS} (without momentum) achieves an iteration complexity of 
	$$O\left(\rho_{A,s}^{-1}\log(\epsilon^{-1})\right).$$ 
	However, Theorem \ref{thm5} reveals that by setting $\omega=1$ and \(\beta=\left(1-\sqrt{0.99\rho_{A,s}} \right)^2\), the iteration complexity of Algorithm \ref{RBKVSm} is significantly reduced to 
	$$O\left(\left(\sqrt{0.99\rho_{A,s}}\right)^{-1}\log(\epsilon^{-1})\right),$$
	indicating a quadratic improvement. Consequently, Theorem \ref{thm5} establishes an accelerated rate for $\lVert \mathbb{E}[x^k-x_*^0] \rVert^2_2$.
\end{remark}

The lemma presented below is crucial for proving Theorem \ref{thm5}.

\begin{lemma}[\cite{fillmore1968linear,elaydi1996introduction}]
	\label{lemma4}
	Consider the second degree linear homogeneous recurrence relation
	$$d_{k+1}=\eta_1 d_k+\eta_2 d_{k-1}$$
	with initial conditions $d_0,d_1\in\mathbb{R}$. Assume that the constant coefficients $\eta_1$ and $\eta_2$ satisfy the inequality $\eta_1^2+4\eta_2<0$. Then there are complex constants $C_0$ and $C_1$ $($depending on the initial conditions $d_0$ and $d_1$$)$ such that:
	$$d_k = 2M^k \left(C_0 \cos(\theta k) + C_1 \sin(\theta k)\right)$$
	where $M = \sqrt{-\eta_2}$ and $\theta$ satisfies that $\eta_1 = 2M\cos \theta$ and
	$\sqrt{-\eta_1^2-4\eta_2}=2M\sin\theta$.
\end{lemma}
\begin{proof}[Proof of Theorem \ref{thm5}]
	From the iteration scheme of Algorithm \ref{RBKVSm}, we have
	\begin{equation*}
		\begin{aligned}
			x^{k+1}-x_*^0&=(I-\omega A_{\mathcal{S}_k}^{\dagger}A_{\mathcal{S}_k})(x^k-x_*^0)+\beta(x^k-x_*^0+x_*^0-x^{k-1})\\
			&=\left((1+\beta)I-\omega A_{\mathcal{S}_k}^{\dagger}A_{\mathcal{S}_k}\right)(x^k-x_*^0)-\beta(x^{k-1}-x_*^0)
		\end{aligned}
	\end{equation*}
	Taking expectations and using Lemma \ref{keylemma-exp1}, we obtain 
	\begin{equation}
		\label{proof-xie-1125-1}
		\mathbb{E}[x^{k+1}-x_*^0]=\left((1+\beta)I-\omega A^{\top}H_s A\right) \mathbb{E}[x^k-x_*^0]-\beta \mathbb{E}[x^{k-1}-x_*^0].
	\end{equation} 
	Let $r:=\operatorname{rank}(A)$ and $A=U\Sigma V^\top$ be the singular value decomposition of $A$, then we know that $A^{\top}H_s A=V\Lambda V^\top$, where
	$$\Lambda= \operatorname{diag}\left(\frac{e_{s-1}(\lambda_{-1})\sigma_{1}^2(A)}{e_s(\lambda)},\ldots,\frac{e_{s-1}(\lambda_{-r})\sigma_{\min}^2(A)}{e_s(\lambda)},0,\ldots,0\right)\in\mathbb{R}^{n\times n}.$$
	By multiplying both sides of \eqref{proof-xie-1125-1}  by $V^\top$, we get
	$$ V^\top\mathbb{E}[x^{k+1}-x_*^0]
	=V^\top\left((1+\beta)I-\omega V\Lambda V^\top\right) \mathbb{E}[x^k-x_*^0]-\beta V^\top\mathbb{E}[x^{k-1}-x_*^0].$$
	Define $z^k:=V^\top\mathbb{E}[x^k-x_*^0]\in\mathbb{R}^n$. The above equation then simplifies to
	$$z^{k+1}=\left( (1+\beta)I-\omega \Lambda \right) z^k-\beta z^{k-1},$$
	which can be expressed coordinate-by-coordinate as
	$$z^{k+1}_i=(1+\beta-\omega \Lambda_{i,i}) z^k_i-\beta z^{k-1}_i, \ \forall \ i=1,2,\dots,n,$$
	where $z^{k}_i$ denotes the $i$-th coordinate of $z^k.$
	
	We now examine two scenarios: $1\leq i\leq r$ and $r+1\leq i\leq n$.
	
	If $r+1\leq i\leq n$, we have $\Lambda_{i,i}=0$, leading to the expression
	$$z^{k+1}_i=(1+\beta)z^k_i-\beta z^{k-1}_i.$$
	Given $x^0-x^0_*=x^1-x^0_*=A^\dagger(Ax^0-b)$, it follows that  $z^{0}_i=z^{1}_i=v_i^\top A^\dagger(Ax^0-b)=0$ since $A^\dagger=V\Sigma^\dagger U^\top$, where $v_i$ denotes the $i$-th column of $V$. Therefore, we derive
	$$z^k_i=0, \ \forall \ k\geq0.$$
	
	If $1\leq i\leq r$,  we can utilize Lemma \ref{lemma4} to establish the desired bound. For any fixed $i$, let $\eta_1:=(1+\beta)-\omega \Lambda_{i,i}$ and $\eta_2:=-\beta$. We need to ensure that $\eta_1^2+4\eta_2<0$. Indeed, since 
	$$\Lambda_{i,i}=\frac{e_{s-1}(\lambda_{-i})\sigma_{i}^2(A)}{e_s(\lambda)}\leq \frac{e_{s-1}(\lambda_{-m})\|A\|^2_2}{e_s(\lambda)},$$
	the assumption $0<\omega \leq \frac{e_{s}(\lambda)}{e_{s-1}(\lambda_{-m})\|A\|^2_2}$ in the theorem implies that $0\leq \omega\Lambda_{i,i}\leq 1$.
	From Lemma \ref{lemma-xie-1125}, we know that
	$$\Lambda_{i,i}=\frac{e_{s-1}(\lambda_{-i})\sigma_{i}^2(A)}{e_s(\lambda)}\geq \frac{\sigma_{i}^2(A)}{\sigma_{\min\{i,s\}}^2(A)+\sum_{j=s+1}^{r}\sigma_j^2(A)}.$$
	If $1\leq i\leq s$, then we have
	$$\Lambda_{i,i}\geq \frac{\sigma_{i}^2(A)}{\sigma_{i}^2(A)+\sum_{j=s+1}^{r}\sigma_j^2(A)}\geq \frac{\sigma_{s}^2(A)}{\sigma_{s}^2(A)+\sum_{j=s+1}^{r}\sigma_j^2(A)}\geq 
	\frac{\sigma_{\min}^2(A)}{\sigma_{s}^2(A)+\sum_{j=s+1}^{r}\sigma_j^2(A)} =\rho_{A,s},$$
	where the second inequality follows from the fact that  $f(x)=\frac{x}{x+\sum_{i=s+1}^{r}\sigma_i^2(A)}$ is an increasing function in $(0,+\infty)$. 
	If $s+1\leq i\leq r$, we also have
	$$\Lambda_{i,i}\geq \frac{\sigma_{i}^2(A)}{\sigma_{s}^2(A)+\sum_{j=s+1}^{r}\sigma_j^2(A)}\geq 
	\frac{\sigma_{\min}^2(A)}{\sigma_{s}^2(A)+\sum_{j=s+1}^{r}\sigma_j^2(A)} =\rho_{A,s}.$$
	Hence, for any fixed $i\in[r]$, we always have
	$\Lambda_{i,i}\geq \rho_{A,s}$.
	Thus, the assumption $ \left(1-\sqrt{\omega\rho_{A,s}}\right)^2 < \beta $ in the theorem implies that
	$$\beta>\left(1-\sqrt{\omega\rho_{A,s}}\right)^2 \geq \left(1-\sqrt{\omega\Lambda_{i,i}}\right)^2.$$
	Therefore, $\eta_1=(1+\beta)-\omega \Lambda_{i,i}\geq 0$ and we have
	$$\eta_1^2+4\eta_2=(1+\beta-\omega \Lambda_{i,i})^2-4\beta \leq (1+\beta-\omega\rho_{A,s})^2-4\beta <0$$
	as desired.
	Applying Lemma \ref{lemma4}, we know that for any $0\leq i\leq r$,
	$$z^k_i = 2(-\eta_2)^{\frac{k}{2}} \left(C_{0,i} \cos(\theta k) + C_{1,i} \sin(\theta k)\right)\leq 2\beta^{\frac{k}{2}}P_i,$$
	where $P_i$ is a constant depending on the initial conditions $C_{0,i},C_{1,i}$ (we can simply choose $P_i = \vert C_{0,i}\vert+\vert C_{1,i}\vert$).
	Combining the two cases together, for all $k\geq0$ we have
	\begin{equation*}
		\begin{aligned}
			\lVert \mathbb{E}[x^k-x_*^0] \rVert^2_2
			&=\lVert V^\top \mathbb{E}[x^k-x_*^0] \rVert^2_2=\lVert z^k \rVert^2_2
			=\sum_{1\leq i\leq r}(z^k_i)^2\leq \sum_{1\leq i\leq r}4\beta^k P_i^2
			=\beta^k c,
		\end{aligned}
	\end{equation*}
	where $c=4\sum_{1\leq i\leq r}P_i^2$.
\end{proof}

\section{Implementation of volume sampling}
\label{vs-eff}

In fact, the computational cost of volume sampling can be quite significant. Despite this, many exact and approximate methods have been proposed for efficiently generating a random variable \(\mathcal{S} \sim \operatorname{Vol}_s (AA^\top)\) \cite{kulesza2012determinantal,anari2016monte,derezinski2024solving}, making it a valuable tool in the design of randomized algorithms.
Alex and Ben \cite[Algorithm 8 and Theorem 5.2]{kulesza2012determinantal} proposed an exact method for volume sampling based on an exact SVD of the matrix $A$ and the mixture of projection DPPs.  However, obtaining an exact SVD can be expensive.

Recently, Derezi\'{n}ski and Yang \cite{derezinski2024solving} demonstrated that approximate volume sampling can be achieved without the SVD of matrix \(A\), by instead utilizing the randomized Hadamard transform. Specifically, the preprocessing stage requires \(O(mn\log m)\) time, and each step takes \(O((s+\log m)\log^3 m)\) to approximately produce a \(\tilde{\mathcal{S}}\) that, with high probability, contains the desired \(\mathcal{S} \). In this section, we present a method to implement volume sampling efficiently and precisely, particularly when the matrix \(AA^\top\) is sparse, inspired by the work in \cite{rodomanov2020randomized}.

\subsection{Preprocessing step}
For the case where \(s=1\), the RBKVS method simplifies to the canonical RK method, which can be readily implemented. Let us consider the case where  $s=2$, i.e., sampling $\{i_0,j_0\}\in \binom{[m]}{s}$. In practice, we can first select one row $i_0$ and then choose the other row $j_0$ with $j_0>i_0$. That is 
\begin{equation}
	\begin{aligned}
		\label{eq-pro}
		\mathbb{P}(i=i_0,j=j_0)
		&=\mathbb{P}(i=i_0)\mathbb{P}(j=j_0\vert i=i_0)\\
		&=\frac{\sum_{j'=i_0+1}^m \operatorname{det}\left(  A_{\{i_0,j'\}}A_{\{i_0,j'\}}^\top\right)}{\sum_{1\leq i' < j'\leq m}\operatorname{det}\left( A_{\{i',j'\}}A_{\{i',j'\}}^\top\right)} 
		\cdot\frac{\operatorname{det}\left( A_{\{i_0,j_0\}}A_{\{i_0,j_0\}}^\top\right)}{\sum_{j'=i_0+1}^m \operatorname{det}\left( A_{\{i_0,j'\}}A_{\{i_0,j'\}}^\top\right)}.
	\end{aligned}
\end{equation}
Note that 
$	\operatorname{det}\left( A_{\{i,j\}}A_{\{i,j\}}^\top\right) 
=\lVert A_i\rVert_2^2 \lVert A_j\rVert_2^2-\langle A_i,A_{j}\rangle^2$. 
Then we have 
\begin{equation*}
	\begin{aligned}
		\label{eq-det}
		\sum_{j'=i+1}^j \operatorname{det}\left( A_{\{i,j'\}}A_{\{i,j'\}}^\top\right)
		&=\sum_{j'=i+1}^j \left(\lVert A_i\rVert_2^2 \lVert A_{j'}\rVert_2^2  -\langle A_i,A_{j'}\rangle^2\right)\\
		&=  \sum_{j'=i}^j \left(\lVert A_i\rVert_2^2\lVert A_{j'}\rVert_2^2  -\langle A_i,A_{j'}\rangle^2\right).
	\end{aligned}
\end{equation*}
For any $1\leq i< j\leq m$, define $\alpha(i,j):=\sum_{j'=i}^j \left(\lVert A_i\rVert_2^2\lVert A_{j'}\rVert_2^2  -\langle A_i,A_{j'}\rangle^2\right)$. Then \eqref{eq-pro} can be simplified as
\begin{equation}
	\label{eq-i,j}
	\mathbb{P}(i=i_0,j=j_0)=\frac{\alpha(i_0,m)}{\sum_{i'=1}^{m-1}\alpha(i',m)}\cdot \frac{\alpha(i_0,j_0)-\alpha(i_0,j_0-1)}{\alpha(i_0,m)},
\end{equation}
Hence, the remaining problem is how to compute $\alpha(i,j),\ 1\leq i< j\leq m$.

First, let us discuss the computational cost of \(AA^\top\).  
For a matrix \(A \in \mathbb{R}^{m \times n}\), assume that none of its rows are zero. 
Define
$$
T=\left\{(i, \ell) \mid A_{i} \circ A_{\ell} \neq 0, 1 \leq i, \ell \leq m \right\},
$$
and for any fixed $i\in[m]$,
$$
T_i=\left\{(i, \ell) \mid A_{i} \circ A_{\ell} \neq 0, 1 \leq \ell \leq m\right\},
$$
where $\circ$ denotes the Hadamard product. From this definition, we know  
$$
T=\cup_{i=1}^m T_i, \ \text{and } \ T_i \cap T_j=\emptyset, \forall i\neq j.
$$
For any $(i,\ell) \in T$, we assume that $A_{i} \circ A_{\ell}$ has $s_{i,\ell}$ nonzero entries. Therefore, the \(i\)-th row of \(A\) contains \(s_{i,i}\) nonzero entries.  
To compute  $AA^\top$, it costs
$$
\sum_{(i, \ell) \in T, i \leq \ell} (2s_{i,\ell}-1)
$$
flops. Note that if \(A\) is a sparse matrix, the cardinality of \(T\), i.e., \(|T|\), can be much smaller than \(m^2\), and in practical cases, it can sometimes be of the order \(O(m)\). Additionally, for \((i, \ell) \in T\), \(s_{i,\ell}\) is typically small, which makes the computation of \(AA^\top\) more efficient.

Next, we discuss the efficient computation of \(\alpha(i,j)\). 
For any fixed \(i \in [m]\), we define
\[
T_{\geq i}=\left\{(i, \ell) \mid A_{i} \circ A_{\ell} \neq 0, i \leq \ell \leq m\right\}.
\]
Thus, \(|T_{\geq i}|\) represents the number of nonzero elements located to the right of the diagonal in the \(i\)-th row of \(AA^\top\). We denote the corresponding column indices as \(j_{i,1}, \ldots, j_{i,|T_{\geq i}|}\), with \(i = j_{i,1}\) since \(A_i\) is nonzero, and \(i = j_{i,1} < \cdots < j_{i,|T_{\geq i}|} \leq m\). Additionally, we define \(j_{i,|T_{\geq i}|+1} = m+1\).
Therefore, for any \(j \in \{i, \ldots, m\}\), there exists a \(k\) with \(1 \leq k \leq |T_{\geq i}|\) such that \(j_{i,k} \leq j \leq j_{i,k+1}-1\). Consequently, \(\sum_{j'=i}^j\langle A_i, A_{j'}\rangle^2 = \sum_{k'=1}^k\langle A_i, A_{j_{i,k'}}\rangle^2\). We can then compute \(\alpha(i,j)\) efficiently with the following equation:
\[
\alpha(i,j)=\lVert A_i\rVert_2^2 \sum_{j'=i}^j \lVert A_{j'}\rVert_2^2-\sum_{k'=1}^k\langle A_i, A_{j_{i,k'}}\rangle^2.
\]
For any \(i \in [m]\) and \(k \in [|T_{\geq i}|]\) satisfying \(j_{i,k} \leq j \leq j_{i,k+1} - 1\), we define \(\zeta_i := \sum_{j=i}^m \lVert A_j\rVert_2^2\) and \(\phi_{i,k} := \sum_{k'=1}^k\langle A_i, A_{j_{i,k'}}\rangle^2\). Thus, we have 
\[
\alpha(i,j) = \lVert A_{i}\rVert_2^2(\zeta_i - \zeta_{j+1}) - \phi_{i,k},
\]
which indicates that our goal now is to compute \(\zeta_i\) and \(\phi_{i,k}\). Note that since we have already computed \(AA^\top\), the computational cost to obtain each \(\zeta_i\) is \(m-1\) flops, and for \(\phi_{i,k}\), it is \(\sum_{i=1}^m(2|T_{\geq i}|-1) = |T|\) flops.  

In Algorithm \ref{preprocess}, we present the preprocessing step for RBKVS with \(s=2\). Note that we also compute \(\alpha(i,j_{i,k})\) for \(i \in [m],1\leq k\leq |T_{\geq i}|\), costing \(\sum_{i=1}^m3|T_{\geq i}|=\frac{3}{2}(|T|+m)\) flops, and \(\alpha(i,m)\) for \(i \in [m]\) with their sums in this algorithm, costing \(3m-1\) flops. From the discussions above and note that 
$$\sum_{(i, \ell) \in T, i \leq \ell} (2s_{i,\ell} - 1)=2\sum_{(i, \ell) \in T, i \leq \ell} s_{i,\ell}- \sum_{i=1}^m |T_{\geq i}|=2\sum_{(i, \ell) \in T, i \leq \ell} s_{i,\ell} - \left(\frac{1}{2}|T| + \frac{1}{2}m\right),$$ we know that the total computational cost of this preprocessing step is given by
\begin{equation*}
	\begin{aligned}
		\sum_{(i, \ell) \in T, i \leq \ell} (2s_{i,\ell} - 1) + m - 1 + |T| + \frac{3}{2}(|T|+m) + 3m - 1 = 2\sum_{(i, \ell) \in T, i \leq \ell} s_{i,\ell} + 2|T| + 5m - 2
	\end{aligned}
\end{equation*}
flops.

\begin{algorithm}
	\caption{Preprocessing step for RBKVS with \(s=2\)}
	\label{preprocess}
	\begin{algorithmic}
		\Require $A\in\mathbb{R}^{m\times n}$ with $\operatorname{rank}(A)\geq2$.
		\begin{enumerate}
			\item[1:] Compute $AA^\top$.
			
			\item[2:] Compute $\zeta:=(\zeta_1,\dots,\zeta_m)^\top\in\mathbb{R}^m$.
			
			\item[3:] Compute $\Phi_i:=(\phi_{i,1},\dots,\phi_{i,|T_{\geq i}|})^\top\in\mathbb{R}^{|T\geq i|},\ i\in[m]$.
			
			\item[4:] Compute $\alpha(i,j_{i,k})$ with $\ i\in[m]$ and $1\leq k\leq |T_{\geq i}|$.
			
			\item[5:] Compute $\Psi\in\mathbb{R}^{m-1}$, where $\Psi_j:=\sum_{i=1}^j \alpha(i,m)$.
		\end{enumerate}
		\Ensure
		$\{\zeta,\Phi_1,\dots,\Phi_m,\alpha(1,j_{1,1}),\dots,\alpha(m,j_{m,|T_{\geq m}|}),\Psi\}$.
	\end{algorithmic}
\end{algorithm}

Finally, we know that for another value of \(s\),  we can apply the same approach to reduce computational costs if \(|T|\) is very small, meaning both \(A\) and \(AA^\top\) are sparse. 
For example, if \(s=3\), we can decompose the probability as \(\mathbb{P}(i=i_0,j=j_0,k=k_0)	=\mathbb{P}(i=i_0)\mathbb{P}(j=j_0\vert i=i_0)\mathbb{P}(k=k_0 \vert j=j_0,i=i_0)\).
However, it should be noted that when \(s\) becomes substantially large, the computational cost of our method may still be high. In such instances, it might be beneficial to consider the alternative techniques proposed in \cite{derezinski2024solving}.

\subsection{Volume sampling via binary search}

From Algorithm \ref{preprocess}, it is can be seen that we compute only a subset of \(\alpha(i,j)\) during the preprocessing step, as computing all \(\alpha(i,j)\) would greatly increase both computational and storage costs. Thus, after selecting \(i_0\), we initially choose a \(k_0\) to refine the search range for \(j_0\) from \(i_0+1 \leq j \leq m\) to \(j_{i_0,k_0} \leq j \leq j_{i_0,k_0+1}-1\).  Using binary search, we demonstrate the process of selecting \(\{i_0,j_0\}\) in Algorithm \ref{sample}, where Step $3$ of the algorithm is the process of selecting \(k_0\). Indeed, the computational cost of Algorithm \ref{sample} is only \(O(\log_2 m)\), which is logarithmic in dimension,  provided that the computation of \(\alpha(i,j)\)  in Step 4 remains low.

\begin{algorithm}
	\caption{Volume sampling for RBKVS with $s=2$}
	\label{sample}
	\begin{algorithmic}
		\Require $\zeta\in\mathbb{R}^m$, $\Phi_1\in\mathbb{R}^{|T\geq 1|},\dots,\Phi_m\in\mathbb{R}^{|T\geq m|}$ and $\Psi\in\mathbb{R}^{m-1}$.
		\begin{enumerate}
			\item[1:] Generate random variables $u_1, u_2\sim\mbox{U}(0,1)$ independently.
			
			\item[2:] Find $i_0:=\min\{1\leq i\leq m-1\mid u_1\leq\frac{\Psi_i}{\Psi_{m-1}}\}$ using binary search.
			
			\item[3:] Find $k_0:=\min\{1\leq k\leq |T_{\geq i_0}| \mid u_2\leq\frac{\alpha(i_0,j_{i_0,k})}{\alpha(i_0,m)}\}$ using binary search.
			
			\item[4:] Compute $\alpha(i_0,j)$ with $ j_{i_0,k_0}\leq j\leq j_{i_0,k_0+1}-1$.
			
			\item[5:] Find $j_0:=\min\{j_{i_0,k_0}\leq j\leq j_{i_0,k_0+1}-1\mid u_2\leq\frac{\alpha(i_0,j)}{\alpha(i_0,m)}\}$ using binary search.
		\end{enumerate}
		\Ensure
		$\{i_0,j_0\}$.
	\end{algorithmic}
\end{algorithm}

\subsection{When to use volume sampling}
\label{sect-when}
In practice, volume sampling can be computationally expensive, particularly due to the preprocessing step. However, if the coefficient matrix \( A \) of the linear system is fixed and only the right-hand side vector \( b \) varies, the preprocessing step (Algorithm \ref{preprocess}) needs to be performed only once. This allows the precomputed results to be reused in subsequent computations, significantly reducing the overall cost. In such cases, the initial overhead of preprocessing is justified, as it enables the efficient repeated use of Algorithm \ref{sample} for volume sampling to solve multiple linear systems with the same \( A \) but different \( b \).

\section{Numerical experiments}
In this section, we implement the proposed RBKVS (Algorithm \ref{RBKVS}) and mRBKVS (Algorithm \ref{RBKVSm}), and compare them with RK, RBK, and the generalized two-subspace randomized Kaczmarz (GTRK) method \cite{wu2022two}. 
All the methods are implemented in MATLAB R2020a for Windows 11 on a LAPTOP PC with an Intel Core i7-1165G7 @ 2.80GHz and 16 GB memory.

Let \(\varpi\) be a uniform random permutation on \([m]\). We consider the following partition for the RBK method
\[
\begin{aligned}
	\mathcal{S}_i &= \left\{\varpi(k): k=(i-1)p+1, (i-1)p+2, \ldots, ip\right\}, \quad i=1, 2, \ldots, t-1, \\
	\mathcal{S}_t &= \left\{\varpi(k): k=(t-1)p+1, (t-1)p+2, \ldots, m\right\}, \quad |\mathcal{S}_t| \leq p,
\end{aligned}
\]
where \(p\) is the block size, and we set \(p = 2\). Let  $\mathcal{S}_k=\{i_k,j_k\}$, where $$
\mathbb{P}(i_k=i)=\frac{\|A_i\|_2^2}{\|A\|^2_F} \ \ \text{and } \ \ \mathbb{P}(j_k=j)=\frac{\|A_j\|_2^2}{\|A\|^2_F-\|A_{i_k}\|_2^2}, 
$$
the  GTRK method has the  iteration scheme
$
x^{k+1}=x^k-A^\dagger_{\mathcal{S}_k }(A_{\mathcal{S}_k }x^k-b_{\mathcal{S}_k}).
$
For RBKVS and mRBKVS, we  set $s=2$.

In our implementations, to ensure the consistency of the linear system, we first generate the solution \( x^* \) and then set \( b = Ax^* \). All computations are initialized with \( x^0 = 0 \). The computations are terminated once the relative solution error (RSE), defined as 
\[
\text{RSE} = \frac{\|x^k - A^\dagger b\|_2^2}{\|x^0 - A^\dagger b\|_2^2},
\]
is less than $10^{-12}$ or the number of iterations exceeds a certain limit. Each experiment is repeated $50$ times, and we calculate the average of the results.

\subsection{Comparison between actual performance and theoretical predictions}
This subsection aims to investigate the actual performance and theoretical predictions \eqref{rate-improved} of the RBKVS method. Note that when \( s = 1 \), the RBKVS method becomes the RK method, and we will compare the performance of RBKVS (\( s = 2 \)) with that of RK.

We generate the coefficient matrix  as follows. 
Given parameters $m, n, r, \sigma_{1}, \sigma_{2}$, and $\delta$, we construct matrices \(A = U D V^\top\), where \(U \in \mathbb{R}^{m \times r}\) and \(V \in \mathbb{R}^{n \times r}\) are column-orthogonal matrices. The diagonal matrices \(D\) have entries defined as:
\[
D_{1,1} =\sigma_{1}, D_{2,2} =\sigma_{2}, \ \text{and} \ D_{i,i}=\delta, i=3,\ldots,r.
\]
Using MATLAB notation, we generate the column-orthogonal matrices with the following commands: {\tt [U,$\sim$]=qr(randn(m,r),0)} and {\tt [V,$\sim$]=qr(randn(n,r),0)}. The exact solution is generated by \(x^* = {\tt randn(n,1)}\) and then set $b=Ax^*$.  We use $x^0 = 0 \in \mathbb{R}^n$ as an initial point. 

We use ``Acc'' to represent the actual acceleration ratio of RBKVS compared to RK in terms of the number of iterations, i.e.,
\[
\operatorname{Acc}:= \frac{\text{Number of iterations of RBKVS}}{\text{Number of iterations of RK}}.
\]
In addition, we use ``PTT'' to 
express the percentage of the ``Acc'' for RBKVS accounting for the theoretical prediction
\eqref{rate-improved}, i.e.,
\[
\text{PTT}:=\frac{\operatorname{Acc}}{\rho_{A,1}/\rho_{A,2}}\times 100=\operatorname{Acc}*\frac{\sum_{i=2}^{\operatorname{rank}(A)} \sigma^2_i(A)}{\sum_{i=1}^{\operatorname{rank}(A)} \sigma^2_i(A)}\times 100.
\]

The numerical experiment results are summarized in Table \ref{table_part1_1} and Table \ref{table_part1_2}, where ``IT'' denotes the number of iterations and ``CPU'' denotes the execution time in seconds. 
The tables reveal that the number of iterations for RK increases considerably with rising \(\sigma_1 / \sigma_2\), whereas RBKVS shows only a slight increase. Moreover, as \(\sigma_1 / \sigma_2\) and \(n\) increase, RBKVS performs increasingly better than the RK method in terms of both the number of iterations and CPU time. Furthermore, the value in the ``PTT'' column
always  around 100, indicating the accuracy of the theoretical prediction \eqref{rate-improved}.

\begin{table}[htbp]
	\footnotesize
	\centering
	\caption{Comparison between actual performance and theoretical predictions with full rank matrices. We set $m=500$, $\sigma_{2}=10$, and  $\delta=0.1$.}
	\begin{tabular}{ccc cc cccc}
		\toprule
		\multicolumn{3}{c}{Parameters} & \multicolumn{2}{c}{RK}  & \multicolumn{4}{c}{RBKVS} \\
		\cmidrule(lr){1-3} \cmidrule(lr){4-5}  \cmidrule(lr){6-9}
		$n$ & $r$ & $\sigma_1/\sigma_2$ & IT & CPU & IT & Acc & PTT & CPU \\
		\midrule
		\multirow{3}{*}{100} 
		
		& 100 & 3 & $1.38\times 10^6$ & 1.3089 & $1.33\times 10^5$ & 10.42 & 105.07 & $4.3221$  \\
		
		& 100 & 9 & $1.13\times 10^7$ & 10.5546 & $1.49\times 10^5$ & 76.18 & 93.81 & $4.6929$  \\
		
		& 100 & 27 & $1.01\times 10^8$ & 95.2760 & $1.48\times 10^5$ & 677.35 & 93.70 & $4.7639$  \\
		
		\midrule
		\multirow{3}{*}{300}  
		&300 & 3 & $1.41\times 10^6$ & 1.7386 & $1.45\times 10^5$ & 9.72 & 99.84 & $4.6408$  \\
		
		&300 & 9 & $1.14\times 10^7$ & 14.1693 & $1.57\times 10^5$ & 72.40 & 90.90 & $5.0613$  \\
		
		&300 & 27 & $1.01\times 10^8$ & 129.8125 & $1.61\times 10^5$ &  624.70 & 88.12 & $5.2437$  \\
		
		\midrule
		\multirow{3}{*}{500} 
		&500 & 3 & $1.42\times 10^6$ & 2.2953 & $1.52\times 10^5$ & 9.38 & 97.93 & $5.0264$  \\
		
		&500  & 9 & $1.14\times 10^7$ & 18.5352 & $1.67\times 10^5$ & 68.23 & 87.30 & $5.5901$  \\
		
		&500  & 27 & $1.01\times 10^8$ & 191.4445 & $1.70\times 10^5$ &  593.55 & 85.35 & $6.0332$  \\
		
		\bottomrule
	\end{tabular}
	\label{table_part1_1}
\end{table}

\begin{table}[htbp]
	\footnotesize
	\centering
	\caption{Comparison between actual performance and theoretical predictions with rank-deficient matrices. We set $m=500$, $\sigma_{2}=10$, and  $\delta=0.1$.}
	\begin{tabular}{ccc cc cccc}
		\toprule
		\multicolumn{3}{c}{Parameters} & \multicolumn{2}{c}{RK}  & \multicolumn{4}{c}{RBKVS} \\
		\cmidrule(lr){1-3} \cmidrule(lr){4-5}  \cmidrule(lr){6-9}
		$n$ & $r$ & $\sigma_1/\sigma_2$ & IT & CPU & IT & Acc & PTT & CPU \\
		\midrule
		\multirow{3}{*}{100} 
		
		& 90 & 3 & $1.38\times 10^6$ & 1.2785 & $1.32\times 10^5$ & 10.45 & 105.37 & $4.0749$  \\
		
		& 90 & 9 & $1.14\times 10^7$ & 10.7777 & $1.47\times 10^5$ & 77.51 & 94.90 & $4.6865$  \\
		
		& 90 & 27 & $1.00\times 10^8$ & 103.0947 & $1.48\times 10^5$ & 681.96 & 94.24 & $4.7571$  \\
		
		\midrule
		\multirow{3}{*}{300}  
		&270 & 3 & $1.40\times 10^6$ & 1.7563 & $1.44\times 10^5$ & 9.75 & 99.84 & $4.7068$  \\
		
		&270 & 9 & $1.14\times 10^7$ &  14.3310 & $1.56\times 10^5$ & 72.72 & 91.03 & $5.0720$  \\
		
		&270 & 27 & $1.01\times 10^8$ & 126.9276 & $1.61\times 10^5$ & 627.61 & 88.28 & $5.1776$  \\
		
		\midrule
		\multirow{3}{*}{500} 
		&450 & 3 & $1.41\times 10^6$ & 2.2769 & $1.52\times 10^5$ & 9.30 & 96.71 & $5.0223$  \\
		
		&450  & 9 & $1.14\times 10^7$ & 18.5580 & $1.64\times 10^5$ & 69.55 & 88.57 & $5.4698$  \\
		
		&450  & 27 & $1.01\times 10^8$ & 166.4904 & $1.69\times 10^5$ &  596.69 & 85.40 & $5.6578$  \\
		
		\bottomrule
	\end{tabular}
	\label{table_part1_2}
\end{table}

\subsection{Efficiency of momentum}

This subsection investigates the effect of heavy ball momentum by comparing the mRBKVS method with its non-momentum variant, RBKVS. We consider two types of coefficient matrices:

{\bf Type I} matrices are generated based on parameters \(m, n, r, \sigma_{1}, \sigma_{2},\) and \(\delta\) as described in the previous subsection.

{\bf Type II} matrices are constructed as follows: Given $m, n, r$, and $\kappa>1$, we construct  matrix $A$ by $A=U D V^\top$, where $U \in \mathbb{R}^{m \times r}, D \in \mathbb{R}^{r \times r}$, and $V \in \mathbb{R}^{n \times r}$. Using {\sc Matlab}  notation, these matrices are generated by {\tt [U,$\sim$]=qr(randn(m,r),0)}, {\tt [V,$\sim$]=qr(randn(n,r),0)}, and {\tt D=diag(1+($\kappa$-1).*rand(r,1))}. So the condition number and the rank of $A$ are upper bounded by $\kappa$ and $r$, respectively.

In our test, the exact solution is generated by \(x^* = {\tt randn(n,1)}\) and then set $b=Ax^*$. We use $x^0 = 0 \in \mathbb{R}^n$ or $x^1 = x^0\in  \mathbb{R}^n$ as an initial point.
We set the stepsize parameter $\omega=1$ for the mRBKVS method. Figures \ref{fig:I} and \ref{fig:II} illustrate the performance of the mRBKVS method with respect to RSE across different momentum parameter values. We note that in all presented tests, the momentum parameters \(\beta\) are chosen to be nonnegative constants that do not depend on unknown parameters such as \(\sigma_{\min}(A)\) and \(\|A\|_2\). 
It can be observed from Figures \ref{fig:I} and \ref{fig:II} that mRBKVS, with appropriately chosen momentum parameters $\beta$, always
converge faster than their non-momentum variant, RBKVS.

\begin{figure}[htbp]
	\centering
	\includegraphics[width=0.32\linewidth]{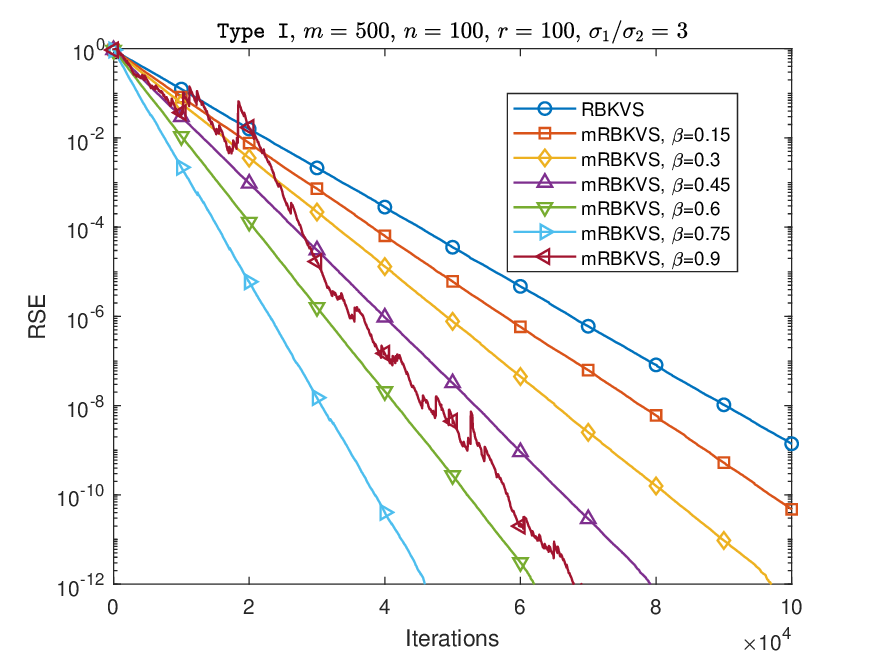}
	\includegraphics[width=0.32\linewidth]{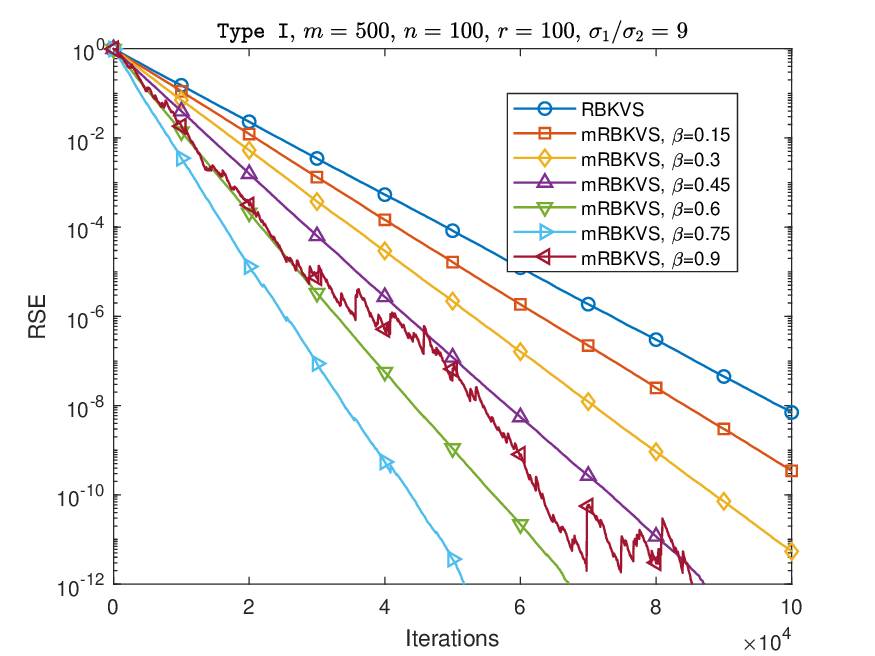}
	\includegraphics[width=0.32\linewidth]{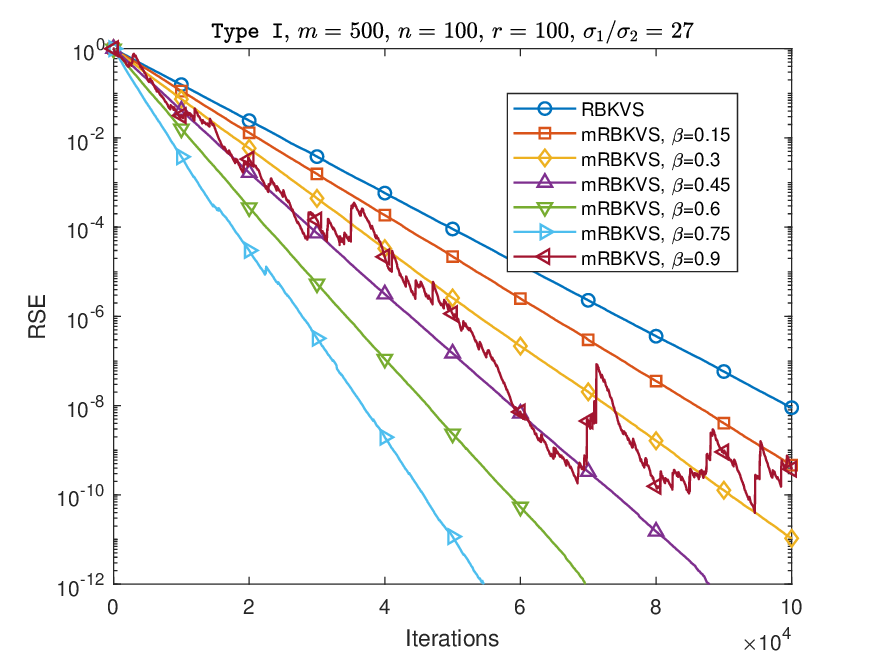}\\
	\includegraphics[width=0.32\linewidth]{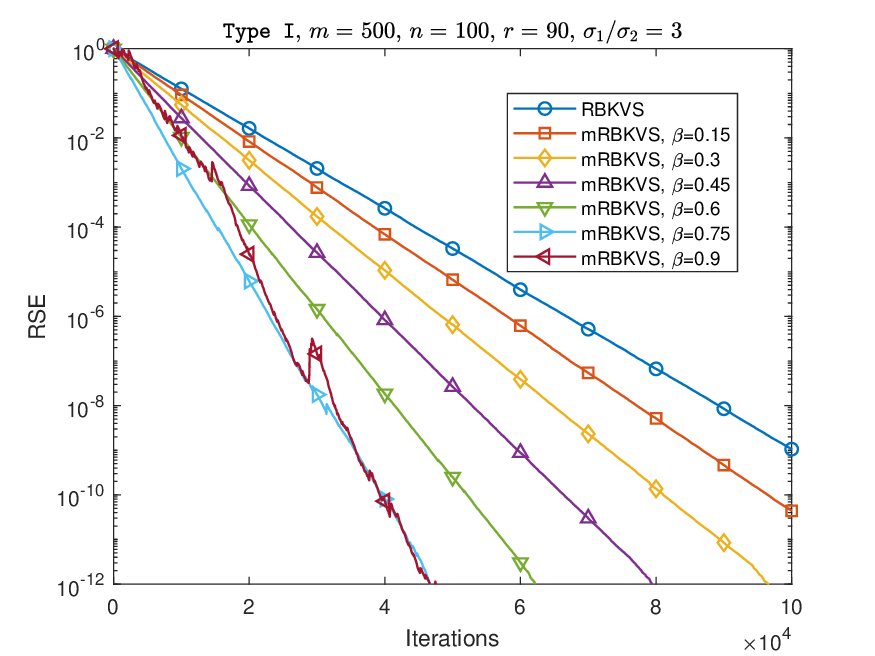}
	\includegraphics[width=0.32\linewidth]{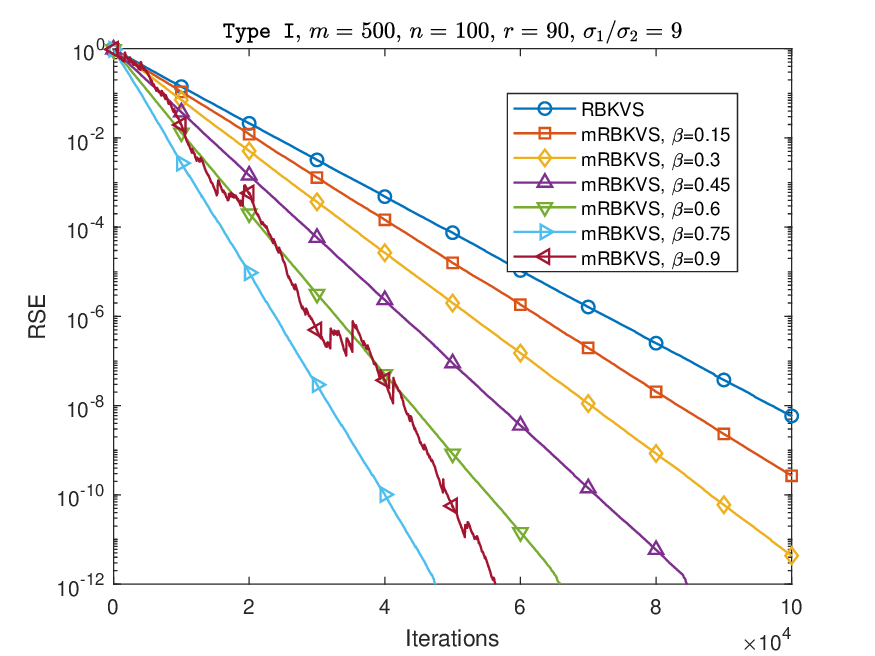}
	\includegraphics[width=0.32\linewidth]{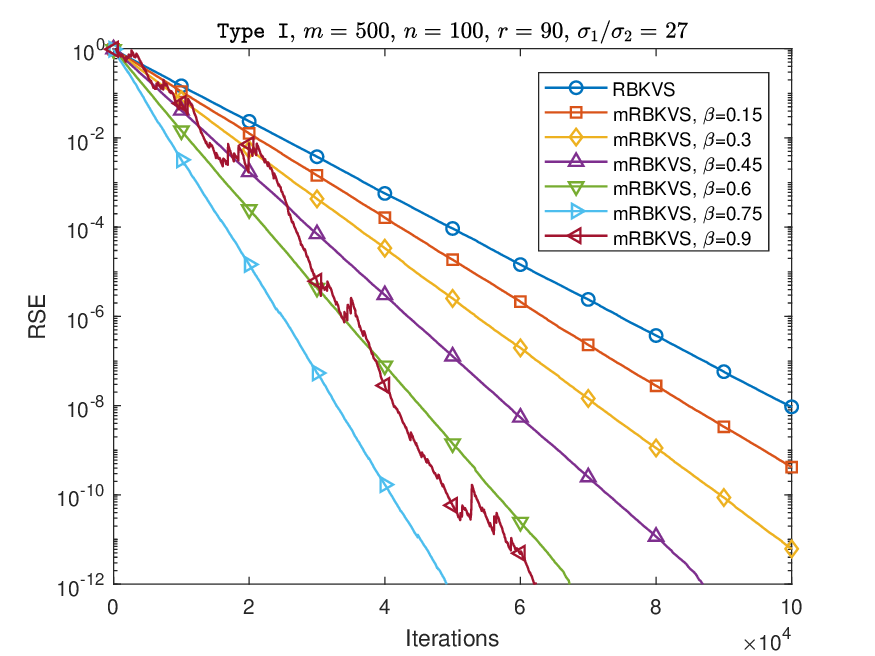}
	\caption{Performance of mRBKVS with different momentum parameters $\beta$ and  Type I coefficient matrices. We set  $\sigma_{2}=10$ and $\delta=0.1$. The title of each plot indicates the values of $m$, $n, r$, and $\sigma_{1}/\sigma_{2}$.}
	\label{fig:I}
\end{figure}

\begin{figure}[htbp]
	\centering
	\includegraphics[width=0.32\linewidth]{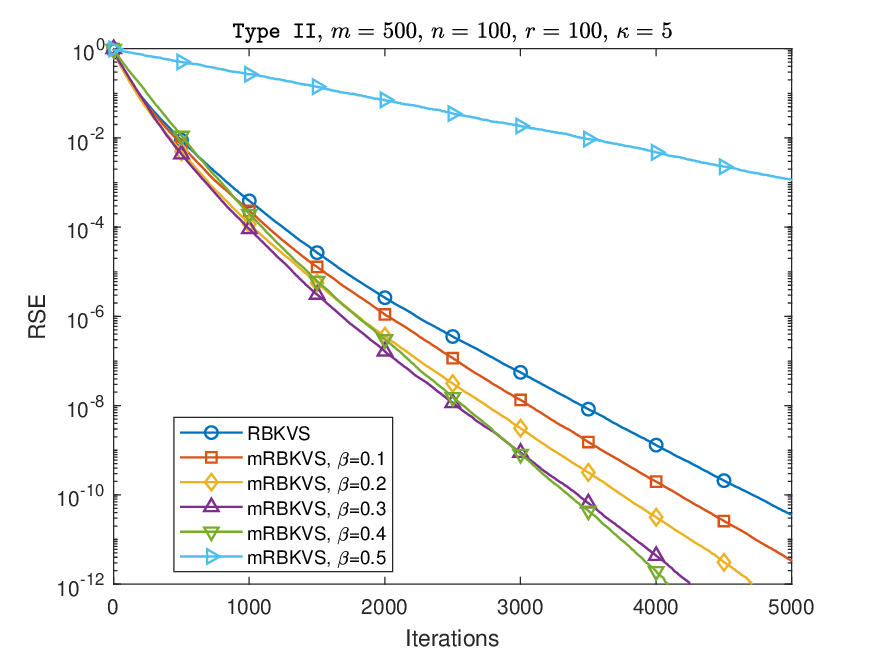}
	\includegraphics[width=0.32\linewidth]{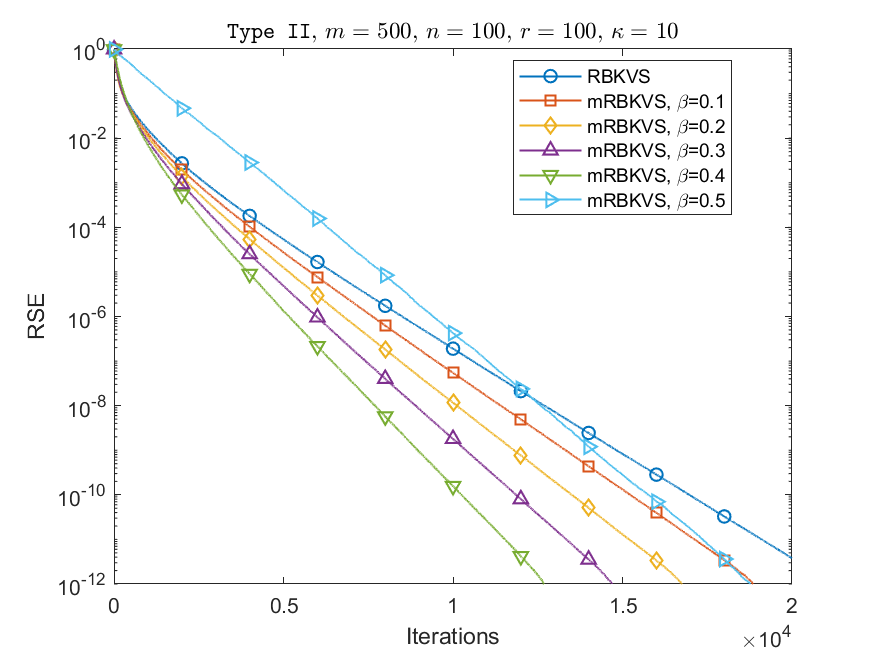}
	\includegraphics[width=0.32\linewidth]{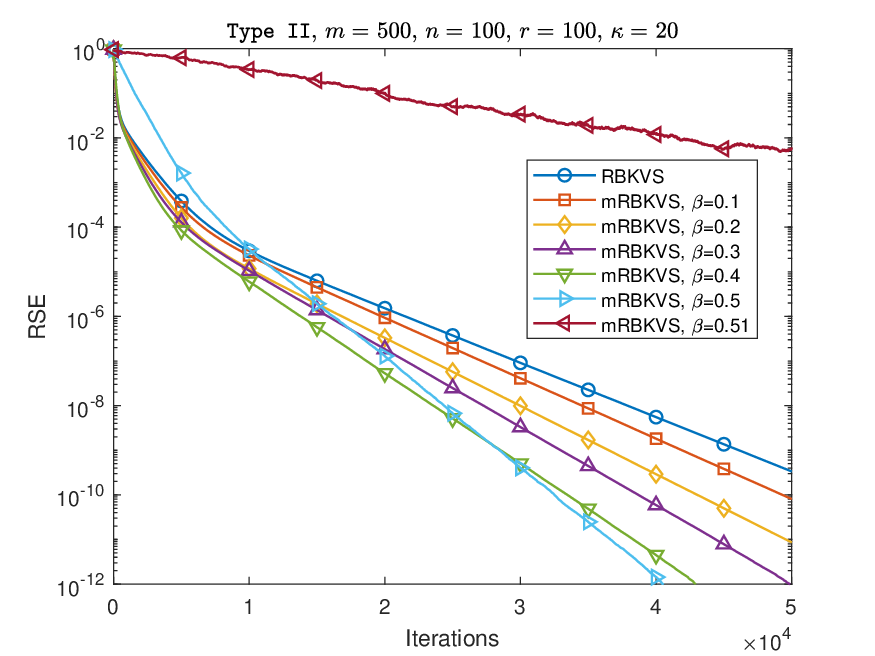}\\
	\includegraphics[width=0.32\linewidth]{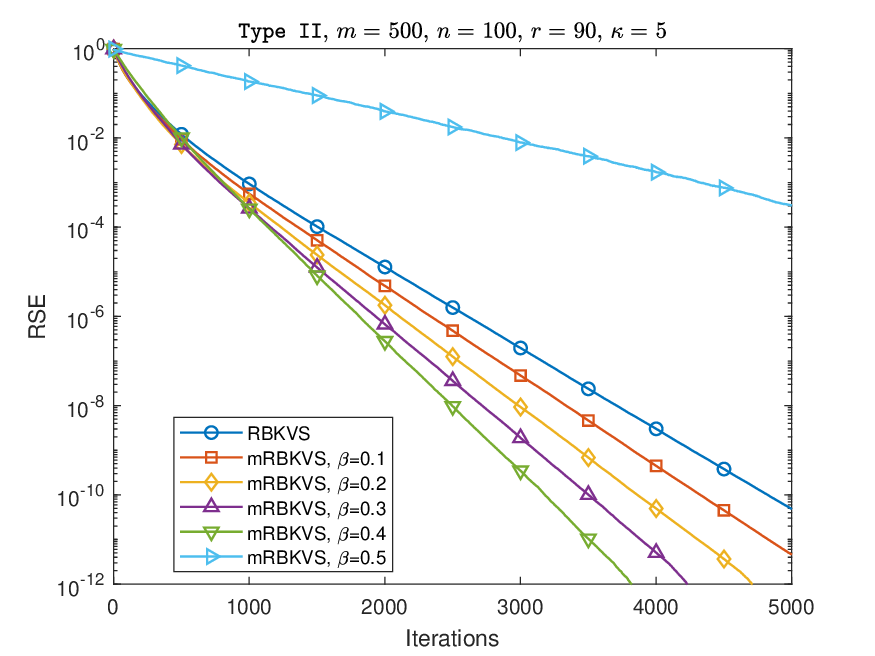}
	\includegraphics[width=0.32\linewidth]{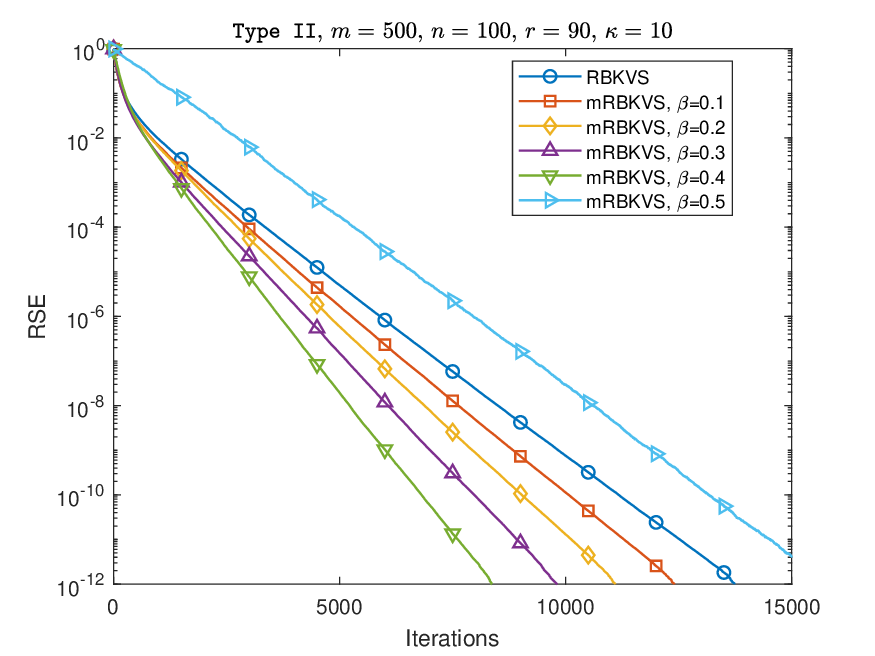}
	\includegraphics[width=0.32\linewidth]{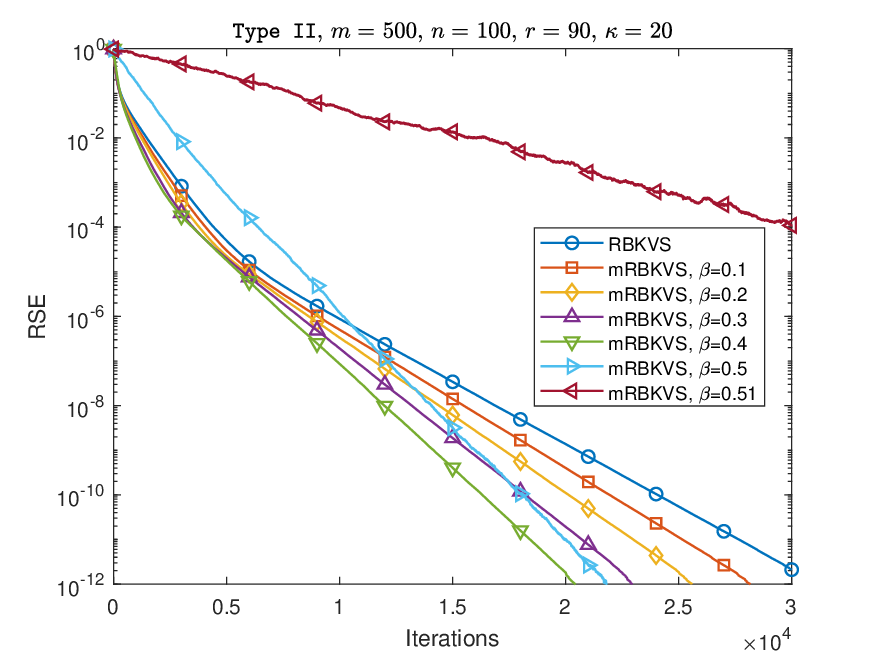}
	\caption{Performance of mRBKVS with different momentum parameters $\beta$ and  Type II coefficient matrices. The title of each plot indicates the values of $m$, $n, r$, and $\kappa$.}
	\label{fig:II}
\end{figure}

\subsection{ Real-world applications} 

In this subsection, we consider the following two types of real-world test instances: the average consensus (AC) problem and the SuiteSparse Matrix Collection  \cite{kolodziej2019suitesparse}. We also compare mRBKVS with RK, RBK, and GTRK. A characteristic of these problems is that the coefficient matrix \( A \) of the linear system is fixed, while only the right-hand side vector \( b \) varies.

\subsubsection{The AC problem}

Let \( G = (V, E) \) be an undirected connected network, where \( V = \{v_1, v_2, \dots, v_n\} \) represents the vertex set and \( E \) (\( |E| = m \)) denotes the edge set. In the average consensus (AC) problem, each vertex \( v_i \in V \) holds a private value \( c_i \in \mathbb{R} \). The objective is to compute the global average \( \bar{c} := \frac{1}{n} \sum_{i=1}^n c_i \) under the constraint that communication is restricted to neighboring vertices. This problem is foundational in distributed computing and multi-agent systems \cite{boyd2006randomized,loizou2021revisiting}, with applications spanning PageRank computation, autonomous agent coordination, and rumor propagation in social networks. Notably, under specific conditions, the randomized pairwise gossip algorithm \cite{boyd2006randomized} for solving the AC problem has been shown to be equivalent to the RK method. Further details can be found in \cite{loizou2021revisiting}.

In our experiments, we utilize two popular graph topologies from the field of wireless sensor networks: the line graph and the cycle graph. Then we consider the homogeneous linear system \( Ax = 0 \), where \( A \in \mathbb{R}^{m \times n} \) is the incidence matrix of the undirected graph.  
The initial values of the nodes are randomly generated using the MATLAB function {\tt rand(n,1)}, and the algorithms aim to compute the average of these values. We set  \( x^0 = c:=(c_1,\ldots,c_n)^\top \) or \( x^1 = x^0 = c \). Finally, we note that with a fixed vertex set, the coefficient matrix \( A \) of the linear system remains constant. This is a typical example of a linear system with a fixed coefficient matrix \( A \), so we only need to perform the preprocessing step on \( A \) at the beginning.

Figure \ref{fig:graph} summarizes the results of the experiment. It can be observed that RBK, GTRK, and RBKVS exhibit nearly similar performance. Although these methods require fewer iterations compared to RK, the CPU time remains almost the same because each step of RK involves less computational cost. The mRBKVS method outperforms all others in both the number of iterations and CPU time.

\begin{figure}[htbp]
	\centering
	\includegraphics[width=0.32\linewidth]{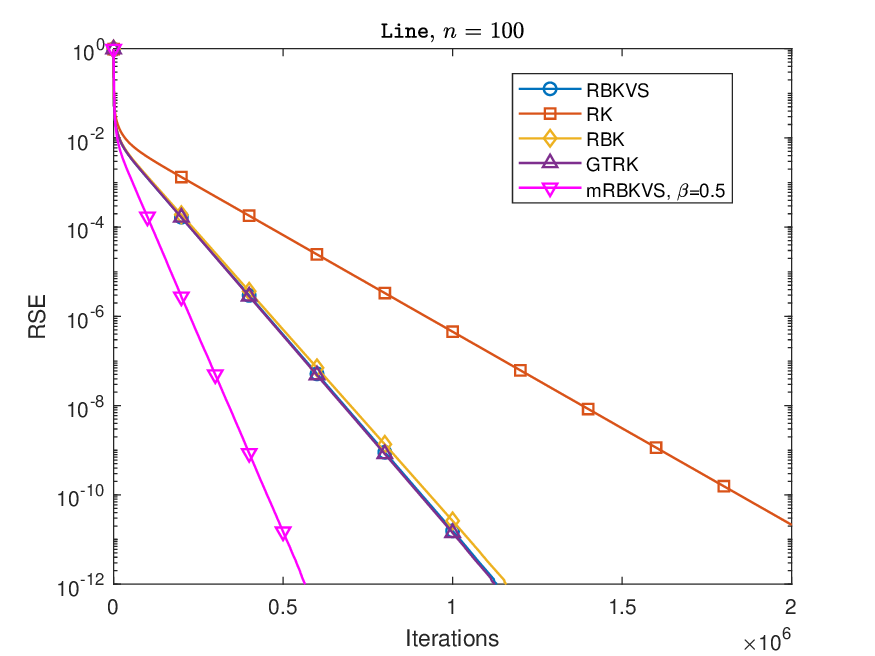}
	\includegraphics[width=0.32\linewidth]{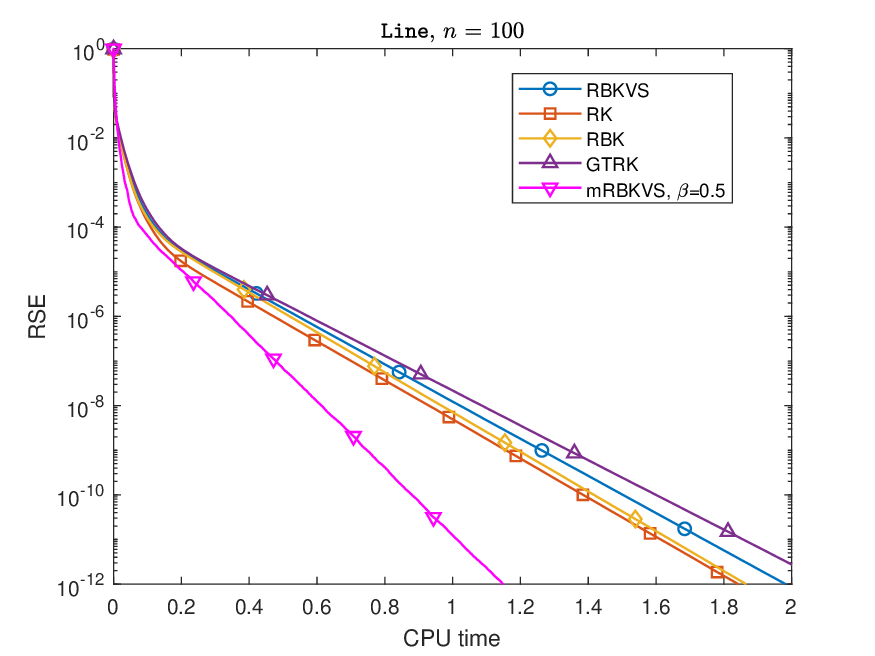}\\
	\includegraphics[width=0.32\linewidth]{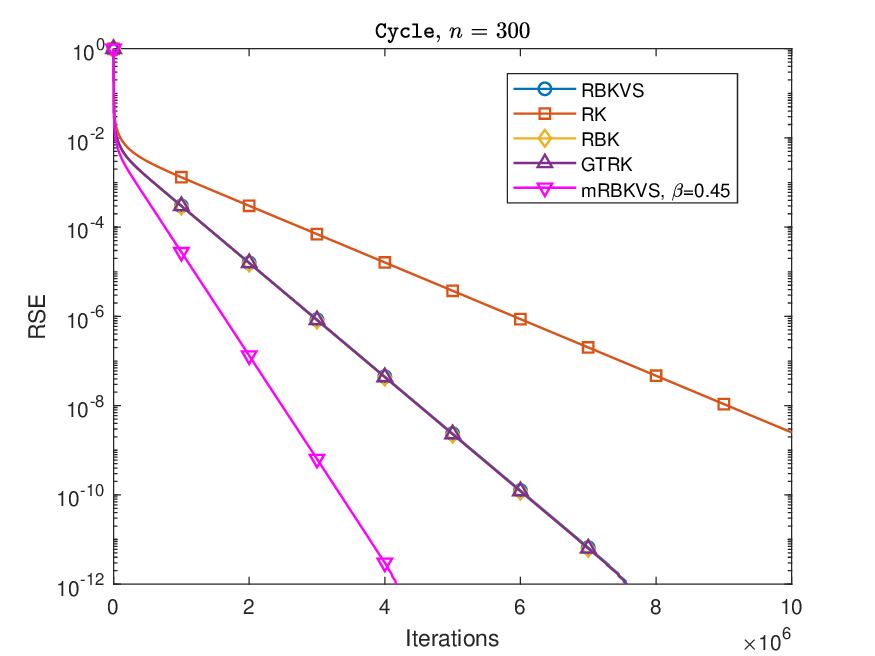}
	\includegraphics[width=0.32\linewidth]{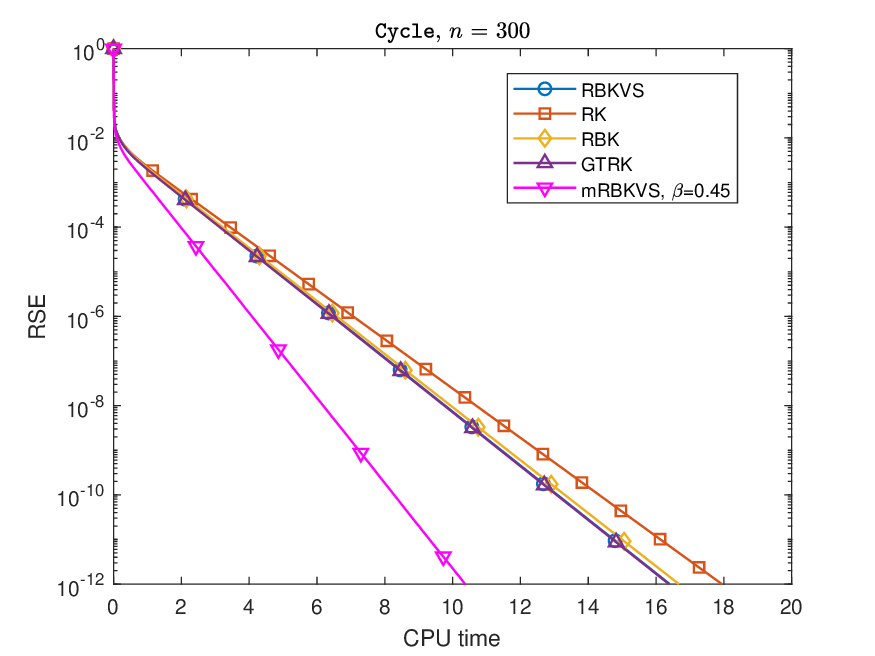}
	\caption{The evolution of RSE with respect to the number of iterations and CPU time.  The CPU times of the preprocessing step for the line graph and the cycle graph are \(0.0660\) and \(1.0923\), respectively. The title of each plot indicates the type of graph and the values of $n$.}
	\label{fig:graph}
\end{figure}

\subsubsection{The SuiteSparse Matrix Collection }

The real-world data are available via the SuiteSparse Matrix Collection \cite{kolodziej2019suitesparse}. The six matrices  
are {\tt lpi\_ex73a}, {\tt flower\_5\_1}, {\tt crew1}, {\tt bibd\_16\_8}, {\tt WorldCities}, and {\tt well1033}; see Table \ref{table_real1}. Each dataset consists of a matrix $A\in\mathbb{R}^{m\times n}$ and a vector $b\in\mathbb{R}^m$. In our experiments, we only use the matrices $A$ of the datasets and ignore the vector $b$. In particular,  the exact solution is generated by \(x^* = {\tt randn(n,1)}\) and then set $b=Ax^*$.

Table \ref{table_real} presents the numerical results, where ``Pre'' denotes the CPU time for the preprocessing step, and the appropriate momentum parameters \(\beta\) for mRBKVS are also listed.
Note that since the matrix \(A\) is provided within each dataset, we only need to perform the preprocessing step on \(A\) at the beginning.
It can be seen that the mRBKVS method outperforms all others in both the number of iterations and CPU time.

\begin{table}[htbp]
	\centering
	\footnotesize
	
	\caption{ The data sets from SuiteSparse Matrix Collection  \cite{kolodziej2019suitesparse}. }
	
	{
		\begin{tabular}{c c c c}
			\toprule
			Matrix & $m\times n$ & rank & $\frac{\sigma_{\max}(A)}{\sigma_{\min}(A)}$  \\

			\midrule
			{\tt lpi\_ex73a} & $193 \times 211$ & 188 & 4.77e16 \\
			
			{\tt flower\_5\_1} & $211 \times 201$ & 179 &  2.00e16 \\
			
			{\tt crew1} & $135 \times 6469$ & 135 &   18.20 \\
			
			{\tt bibd\_16\_8} & $120 \times 12870$ & 120 & 9.54 \\
			
			{\tt WorldCities} & $315 \times 100$ & 100 &   6.60 \\
			
			{\tt well1033} & $1033 \times 320$ & 320& 166.13 \\
			\bottomrule
		\end{tabular}
	}
	\label{table_real1}
\end{table}

\begin{table}[htbp]
	\footnotesize
	\centering
	
	\caption{Performance of RK, RBK, GTRK, and mRBKVS for linear systems with coefficient matrices from the SuiteSparse Matrix Collection \cite{kolodziej2019suitesparse}. }
	
	{\scriptsize
		\begin{tabular}{c cc cc cc cccc}
			\toprule
			\multirow{2}{*}{Matrix}  & \multicolumn{2}{c}{RK} & \multicolumn{2}{c}{RBK} & \multicolumn{2}{c}{GTRK} & \multicolumn{4}{c}{mRBKVS} \\
			\cmidrule(lr){2-3} \cmidrule(lr){4-5} \cmidrule(lr){6-7} \cmidrule(lr){8-11}
			&  IT & CPU & IT & CPU & IT & CPU & Pre &$\beta$ & IT & CPU \\
			\midrule
			{\tt lpi\_ex73a}  & $1.19\text{e}6$ &  6.6634 & $5.97\text{e}5$  & 10.9108 & $5.91\text{e}5$ & 10.7430 &  0.3221 & 0.50 & $2.95\text{e}5$ & \textbf{5.4576}\\
			
			{\tt flower\_5\_1}  & $1.03\text{e}5$ & 0.6295 & $4.20\text{e}4$  & 0.7812 & $5.16\text{e}4$ & 0.9477 & 0.4062 & 0.45 & $2.80\text{e}4$ & \textbf{0.5272}\\
			
			{\tt crew1}   & $3.39\text{e}4$ & 3.9654 & $1.55\text{e}4$  & 3.6824 & $1.70\text{e}4$ & 4.1246 & 1.8475 & 0.40 & $1.05\text{e}4$ & \textbf{2.5441}\\
			
			{\tt bibd\_16\_8}  & $6.58\text{e}3$ & 2.2510 &  $3.78\text{e}3$ & 2.5061 & $3.11\text{e}3$ &  2.0814 & 3.0306 & 0.25 & $2.90\text{e}3$ & \textbf{1.9807}        \\

			{\tt WorldCities} & $8.87\text{e}4$ & 0.7327 &  $2.00\text{e}8$ & $8.6040\text{e}3$ & $3.58\text{e}4$ & 0.8464 & 1.3506 & 0.50 & $1.62\text{e}4$ & \textbf{0.3884}        \\
			
			{\tt well1033}  & $3.02\text{e}7$ & 275.6034 &  $1.02\text{e}7$ & 226.7555 & $1.44\text{e}7$ & 336.4390 & 516.1764 & 0.50 & $7.70\text{e}6$ & \textbf{200.4543}        \\
			\bottomrule
		\end{tabular}
	}
	\label{table_real}
\end{table}

\section{Concluding remarks}

In this work, we studied the RBK method enhanced with volume sampling and heavy ball momentum for solving linear systems. We proved global linear convergence rates for the method and demonstrated an accelerated linear rate in terms of the norm of the expected error. Our theoretical findings were corroborated through extensive experimental testing, confirming the superior performance of the mRBKVS method.

There are still many possible future venues of research. The randomized extended Kaczmarz method \cite{Du20Ran,Zou12,Du19} is a highly effective approach for solving inconsistent linear systems. Investigating how to leverage RBKVS to design a corresponding randomized extended block Kaczmarz method is a worthwhile research direction. Recent studies \cite{rieger2023generalized, hegland2023generalized} have shown that RK-type methods can be accelerated by the Gearhart-Koshy acceleration \cite{gearhart1989acceleration, tam2021gearhart}.
Exploring extensions of Gearhart-Koshy acceleration for the RBKVS method could also be a valuable research topic.

		\bibliographystyle{plain}   
		\bibliography{bib}

\begin{thebibliography}{10}

\bibitem{anari2016monte}
Nima Anari, Shayan~Oveis Gharan, and Alireza Rezaei.
\newblock Monte {C}arlo {M}arkov chain algorithms for sampling strongly
  {R}ayleigh distributions and determinantal point processes.
\newblock In {\em Conference on Learning Theory}, pages 103--115. PMLR, 2016.

\bibitem{anari2024optimal}
Nima Anari, Yang~P Liu, and Thuy-Duong Vuong.
\newblock Optimal sublinear sampling of spanning trees and determinantal point
  processes via average-case entropic independence.
\newblock {\em SIAM Journal on Computing}, pages FOCS22--93, 2024.

\bibitem{avron2013faster}
Haim Avron and Christos Boutsidis.
\newblock Faster subset selection for matrices and applications.
\newblock {\em SIAM Journal on Matrix Analysis and Applications},
  34(4):1464--1499, 2013.

\bibitem{bai2018greedy}
Zhong-Zhi Bai and Wen-Ting Wu.
\newblock On greedy randomized {K}aczmarz method for solving large sparse
  linear systems.
\newblock {\em SIAM Journal on Scientific Computing}, 40(1):A592--A606, 2018.

\bibitem{bai2023randomized}
Zhong-Zhi Bai and Wen-Ting Wu.
\newblock Randomized {K}aczmarz iteration methods: Algorithmic extensions and
  convergence theory.
\newblock {\em Japan Journal of Industrial and Applied Mathematics},
  40(3):1421--1443, 2023.

\bibitem{bernstein2009matrix}
Dennis~S Bernstein.
\newblock {\em Matrix mathematics: theory, facts, and formulas}.
\newblock Princeton University Press, 2009.

\bibitem{boutsidis2014near}
Christos Boutsidis, Petros Drineas, and Malik Magdon-Ismail.
\newblock Near-optimal column-based matrix reconstruction.
\newblock {\em SIAM Journal on Computing}, 43(2):687--717, 2014.

\bibitem{boyd2006randomized}
Stephen Boyd, Arpita Ghosh, Balaji Prabhakar, and Devavrat Shah.
\newblock Randomized gossip algorithms.
\newblock {\em IEEE Trans. Inform. Theory}, 52(6):2508--2530, 2006.

\bibitem{byrne2003unified}
Charles Byrne.
\newblock A unified treatment of some iterative algorithms in signal processing
  and image reconstruction.
\newblock {\em Inverse problems}, 20(1):103, 2003.

\bibitem{cai2024interlacing}
Jian-Feng Cai, Zhiqiang Xu, and Zili Xu.
\newblock Interlacing polynomial method for the column subset selection
  problem.
\newblock {\em International Mathematics Research Notices}, 2024(9):7798--7819,
  2024.

\bibitem{censor1981row}
Yair Censor.
\newblock Row-action methods for huge and sparse systems and their
  applications.
\newblock {\em SIAM review}, 23(4):444--466, 1981.

\bibitem{derezinski2021determinantal}
Micha{\l} Derezinski and Michael~W Mahoney.
\newblock Determinantal point processes in randomized numerical linear algebra.
\newblock {\em Notices of the American Mathematical Society}, 68(1):34--45,
  2021.

\bibitem{derezinski2018reverse}
Micha{\l} Derezi{\'n}ski and Manfred~K Warmuth.
\newblock Reverse iterative volume sampling for linear regression.
\newblock {\em Journal of Machine Learning Research}, 19(23):1--39, 2018.

\bibitem{derezinski2024solving}
Micha{\l} Derezi{\'n}ski and Jiaming Yang.
\newblock Solving dense linear systems faster than via preconditioning.
\newblock In {\em Proceedings of the 56th Annual ACM Symposium on Theory of
  Computing}, pages 1118--1129, 2024.

\bibitem{deshpande2006matrix}
Amit Deshpande, Luis Rademacher, Santosh~S Vempala, and Grant Wang.
\newblock Matrix approximation and projective clustering via volume sampling.
\newblock {\em Theory of Computing}, 2(1):225--247, 2006.

\bibitem{Du19}
Kui Du.
\newblock Tight upper bounds for the convergence of the randomized extended
  {K}aczmarz and {G}auss-{S}eidel algorithms.
\newblock {\em Numer. Linear Algebra Appl.}, 26(3):e2233, 2019.

\bibitem{Du20Ran}
Kui Du, Wu-Tao Si, and Xiao-Hui Sun.
\newblock Randomized extended average block {K}aczmarz for solving least
  squares.
\newblock {\em SIAM J. Sci. Comput.}, 42(6):A3541--A3559, 2020.

\bibitem{elaydi1996introduction}
Saber Elaydi.
\newblock {\em An introduction to difference equations}.
\newblock Springer New York, NY, 1996.

\bibitem{elfving1980block}
Tommy Elfving.
\newblock Block-iterative methods for consistent and inconsistent linear
  equations.
\newblock {\em Numerische Mathematik}, 35:1--12, 1980.

\bibitem{ferreira2024parallelization}
In{\^e}s Ferreira, Juan~A Acebr{\'o}n, and Jos{\'e} Monteiro.
\newblock Parallelization strategies for the randomized {K}aczmarz algorithm on
  large-scale dense systems.
\newblock {\em arXiv preprint arXiv:2401.17474}, 2024.

\bibitem{a2024survey}
In{\^e}s Ferreira, Juan~A Acebr{\'o}n, and Jos{\'e} Monteiro.
\newblock Survey of a class of iterative row-action methods: The {K}aczmarz
  method.
\newblock {\em Numerical Algorithms}, pages 1--39, 2024.

\bibitem{fillmore1968linear}
Jay~P Fillmore and Morris~L Marx.
\newblock Linear recursive sequences.
\newblock {\em SIAM Review}, 10(3):342--353, 1968.

\bibitem{garrigos2023handbook}
Guillaume Garrigos and Robert~M Gower.
\newblock Handbook of convergence theorems for (stochastic) gradient methods.
\newblock {\em arXiv preprint arXiv:2301.11235}, 2023.

\bibitem{gearhart1989acceleration}
William~B Gearhart and Mathew Koshy.
\newblock Acceleration schemes for the method of alternating projections.
\newblock {\em Journal of Computational and Applied Mathematics},
  26(3):235--249, 1989.

\bibitem{ghadimi2015global}
Euhanna Ghadimi, Hamid~Reza Feyzmahdavian, and Mikael Johansson.
\newblock Global convergence of the heavy-ball method for convex optimization.
\newblock In {\em 2015 European control conference (ECC)}, pages 310--315.
  IEEE, 2015.

\bibitem{gower2024bregman}
Robert Gower, Dirk~A Lorenz, and Maximilian Winkler.
\newblock A {Bregman--Kaczmarz} method for nonlinear systems of equations.
\newblock {\em Computational Optimization and Applications}, 87(3):1059--1098,
  2024.

\bibitem{Gow15}
Robert~M. Gower and Peter Richt{\'a}rik.
\newblock Randomized iterative methods for linear systems.
\newblock {\em SIAM J. Matrix Anal. Appl.}, 36(4):1660--1690, 2015.

\bibitem{9869773}
Jamie Haddock, Benjamin Jarman, and Chen Yap.
\newblock Paving the way for consensus: Convergence of block gossip algorithms.
\newblock {\em IEEE Transactions on Information Theory}, 68(11):7515--7527,
  2022.

\bibitem{Han2022-xh}
Deren Han, Yansheng Su, and Jiaxin Xie.
\newblock Randomized {Douglas-Rachford} methods for linear systems: {I}mproved
  accuracy and efficiency.
\newblock {\em SIAM J. Optim.}, 34(1):1045--1070, 2024.

\bibitem{han2022pseudoinverse}
Deren Han and Jiaxin Xie.
\newblock On pseudoinverse-free randomized methods for linear systems:
  {U}nified framework and acceleration.
\newblock {\em arXiv preprint arXiv:2208.05437}, 2022.

\bibitem{hefny2017rows}
Ahmed Hefny, Deanna Needell, and Aaditya Ramdas.
\newblock Rows versus columns: Randomized {K}aczmarz or {Gauss--Seidel} for
  ridge regression.
\newblock {\em SIAM Journal on Scientific Computing}, 39(5):S528--S542, 2017.

\bibitem{hegland2023generalized}
Markus Hegland and Janosch Rieger.
\newblock Generalized {G}earhart-{K}oshy acceleration is a {K}rylov space
  method of a new type.
\newblock {\em arXiv preprint arXiv:2311.18305}, 2023.

\bibitem{herman2009fundamentals}
Gabor~T Herman.
\newblock {\em Fundamentals of computerized tomography: image reconstruction
  from projections}.
\newblock Springer Science \& Business Media, 2009.

\bibitem{herman2008image}
Gabor~T Herman and Ran Davidi.
\newblock Image reconstruction from a small number of projections.
\newblock {\em Inverse problems}, 24(4):045011, 2008.

\bibitem{karczmarz1937angenaherte}
Stefan Karczmarz.
\newblock Angenaherte auflosung von systemen linearer glei-chungen.
\newblock {\em Bull. Int. Acad. Pol. Sic. Let., Cl. Sci. Math. Nat.}, pages
  355--357, 1937.

\bibitem{kolodziej2019suitesparse}
Scott~P Kolodziej, Mohsen Aznaveh, Matthew Bullock, Jarrett David, Timothy~A
  Davis, Matthew Henderson, Yifan Hu, and Read Sandstrom.
\newblock The suitesparse matrix collection website interface.
\newblock {\em Journal of Open Source Software}, 4(35):1244, 2019.

\bibitem{kulesza2012determinantal}
Alex Kulesza and Ben Taskar.
\newblock Determinantal point processes for machine learning.
\newblock {\em Foundations and Trends{\textregistered} in Machine Learning},
  5(2--3):123--286, 2012.

\bibitem{liu2016accelerated}
Ji~Liu and Stephen Wright.
\newblock An accelerated randomized {K}aczmarz algorithm.
\newblock {\em Math. Comp.}, 85(297):153--178, 2016.

\bibitem{liu2014asynchronous}
Ji~Liu, Stephen~J Wright, and Srikrishna Sridhar.
\newblock An asynchronous parallel randomized {K}aczmarz algorithm.
\newblock {\em arXiv preprint arXiv:1401.4780}, 2014.

\bibitem{loizou2020momentum}
Nicolas Loizou and Peter Richt{\'a}rik.
\newblock Momentum and stochastic momentum for stochastic gradient, newton,
  proximal point and subspace descent methods.
\newblock {\em Computational Optimization and Applications}, 77(3):653--710,
  2020.

\bibitem{loizou2021revisiting}
Nicolas Loizou and Peter Richt{\'a}rik.
\newblock Revisiting randomized gossip algorithms: General framework,
  convergence rates and novel block and accelerated protocols.
\newblock {\em IEEE Transactions on Information Theory}, 67(12):8300--8324,
  2021.

\bibitem{lok2024subspace}
Jackie Lok and Elizaveta Rebrova.
\newblock A subspace constrained randomized {K}aczmarz method for structure or
  external knowledge exploitation.
\newblock {\em Linear Algebra and its Applications}, 2024.

\bibitem{macchi1975coincidence}
Odile Macchi.
\newblock The coincidence approach to stochastic point processes.
\newblock {\em Advances in Applied Probability}, 7(1):83--122, 1975.

\bibitem{natterer2001mathematics}
Frank Natterer.
\newblock {\em The mathematics of computerized tomography}.
\newblock SIAM, 2001.

\bibitem{Nec19}
Ion Necoara.
\newblock Faster randomized block {K}aczmarz algorithms.
\newblock {\em SIAM J. Matrix Anal. Appl.}, 40(4):1425--1452, 2019.

\bibitem{needell2010randomized}
Deanna Needell.
\newblock Randomized {K}aczmarz solver for noisy linear systems.
\newblock {\em BIT Numerical Mathematics}, 50:395--403, 2010.

\bibitem{needell2014paved}
Deanna Needell and Joel~A Tropp.
\newblock Paved with good intentions: analysis of a randomized block {K}aczmarz
  method.
\newblock {\em Linear Algebra and its Applications}, 441:199--221, 2014.

\bibitem{polyak1964some}
Boris~T Polyak.
\newblock Some methods of speeding up the convergence of iteration methods.
\newblock {\em Ussr Computational Mathematics and Mathematical Physics},
  4(5):1--17, 1964.

\bibitem{popa2004kaczmarz}
Constantin Popa and Rafal Zdunek.
\newblock Kaczmarz extended algorithm for tomographic image reconstruction from
  limited-data.
\newblock {\em Mathematics and Computers in Simulation}, 65(6):579--598, 2004.

\bibitem{rieger2023generalized}
Janosch Rieger.
\newblock Generalized {G}earhart-{K}oshy acceleration for the {K}aczmarz
  method.
\newblock {\em Mathematics of Computation}, 92(341):1251--1272, 2023.

\bibitem{robbins1951stochastic}
Herbert Robbins and Sutton Monro.
\newblock A stochastic approximation method.
\newblock {\em The Annals of Mathematical Statistics}, pages 400--407, 1951.

\bibitem{rodomanov2020randomized}
Anton Rodomanov and Dmitry Kropotov.
\newblock A randomized coordinate descent method with volume sampling.
\newblock {\em SIAM Journal on Optimization}, 30(3):1878--1904, 2020.

\bibitem{schopfer2019linear}
Frank Sch{\"o}pfer and Dirk~A Lorenz.
\newblock Linear convergence of the randomized sparse {K}aczmarz method.
\newblock {\em Math. Program.}, 173(1):509--536, 2019.

\bibitem{sebbouh2021almost}
Othmane Sebbouh, Robert~M Gower, and Aaron Defazio.
\newblock Almost sure convergence rates for stochastic gradient descent and
  stochastic heavy ball.
\newblock In {\em Conference on Learning Theory}, pages 3935--3971. PMLR, 2021.

\bibitem{strohmer2009randomized}
Thomas Strohmer and Roman Vershynin.
\newblock A randomized {K}aczmarz algorithm with exponential convergence.
\newblock {\em Journal of Fourier Analysis and Applications}, 15(2):262--278,
  2009.

\bibitem{su2024greedy}
Yansheng Su, Deren Han, Yun Zeng, and Jiaxin Xie.
\newblock On greedy multi-step inertial randomized {K}aczmarz method for
  solving linear systems.
\newblock {\em Calcolo}, 61(4):68, 2024.

\bibitem{tam2021gearhart}
Matthew~K Tam.
\newblock Gearhart--{K}oshy acceleration for affine subspaces.
\newblock {\em Operations Research Letters}, 49(2):157--163, 2021.

\bibitem{tondji2024acceleration}
Lionel Tondji, Ion Necoara, and Dirk~A Lorenz.
\newblock Acceleration and restart for the randomized {Bregman-Kaczmarz}
  method.
\newblock {\em Linear Algebra and its Applications}, 699:508--538, 2024.

\bibitem{wang2022nonlinear}
Qifeng Wang, Weiguo Li, Wendi Bao, and Xingqi Gao.
\newblock Nonlinear {K}aczmarz algorithms and their convergence.
\newblock {\em Journal of Computational and Applied Mathematics}, 399:113720,
  2022.

\bibitem{wu2022two}
Wen-Ting Wu.
\newblock On two-subspace randomized extended {K}aczmarz method for solving
  large linear least-squares problems.
\newblock {\em Numerical Algorithms}, 89(1):1--31, 2022.

\bibitem{yuan2022adaptively}
Zi-Yang Yuan, Lu~Zhang, Hongxia Wang, and Hui Zhang.
\newblock Adaptively sketched {B}regman projection methods for linear systems.
\newblock {\em Inverse Problems}, 38(6):065005, 2022.

\bibitem{zeng2023fast}
Yun Zeng, Deren Han, Yansheng Su, and Jiaxin Xie.
\newblock Fast stochastic dual coordinate descent algorithms for linearly
  constrained convex optimization.
\newblock {\em arXiv preprint arXiv:2307.16702}, 2023.

\bibitem{zeng2024adaptive}
Yun Zeng, Deren Han, Yansheng Su, and Jiaxin Xie.
\newblock On adaptive stochastic heavy ball momentum for solving linear
  systems.
\newblock {\em SIAM Journal on Matrix Analysis and Applications},
  45(3):1259--1286, 2024.

\bibitem{Zou12}
Anastasios Zouzias and Nikolaos~M. Freris.
\newblock Randomized extended {K}aczmarz for solving least squares.
\newblock {\em SIAM J. Matrix Anal. Appl.}, 34(2):773--793, 2013.

\end{thebibliography}
		
\section{Appendix. Proof of the key lemmas}
\label{sec:appd}

\subsection{Proof of Lemmas \ref{keylemma-exp1} and \ref{keylemma-exp2}}		

To establish Lemmas \ref{keylemma-exp1} and \ref{keylemma-exp2}, we first introduce some useful lemmas.

\begin{lemma}[\cite{rodomanov2020randomized}, Lemma 3.3]
	\label{lemmaH}
	Under the same notations of Lemma \ref{keylemma-exp1}, we have
	\[
	\sum_{\mathcal{S} \in \binom{[m]}{s}} I_{\mathcal{S} }^\top \operatorname{Adj}(A_{\mathcal{S} } A_{\mathcal{S} }^\top) I_{\mathcal{S} } = U \operatorname{diag}\left(e_{s-1}(\lambda_{-1}), \ldots, e_{s-1}(\lambda_{-m})\right) U^\top.
	\]
\end{lemma}

The following result, which can be derived from the well-known Cauchy-Binet identity, provides a concise expression for the sum of principal minors.
\begin{lemma}[\cite{bernstein2009matrix}, Facts 7.5.17, (\romannumeral18)]
	\label{lemma-spm}
	Under the same notations of Lemma \ref{keylemma-exp1}, for any integer $s$ such that $1\leq s \leq m$, we have
	$$\sum_{\mathcal{S}\in\binom{[m]}{s}}\operatorname{det}(A_{\mathcal{S} } A_{\mathcal{S} }^\top)=e_{s}(\lambda).$$
\end{lemma}

\begin{lemma}
	\label{lemma0}
	For any matrix $B\in\mathbb{R}^{\ell \times n}$, the matrix $\operatorname{Adj}(BB^\top)$ is positive semidefinite. Moreover, $\operatorname{rank}(B)<\ell$ if and only if $B^\top\operatorname{Adj}(BB^\top)B=0$. 
\end{lemma}
\begin{proof}
	Let $B=U\Sigma V^\top$ be the singular value decomposition of $B$. Note that for any square matrices $P$ and $Q$,  $\operatorname{Adj}(PQ)=\operatorname{Adj}(Q)\operatorname{Adj}(P)$. Thus, we obtain
	$$\operatorname{Adj}(BB^\top)=\operatorname{Adj}(U \Sigma\Sigma^\top U^\top)=\operatorname{Adj}(U^\top)\operatorname{Adj}(\Sigma\Sigma^\top )\operatorname{Adj}(U)=U\operatorname{Adj}(\Sigma\Sigma^\top )U^\top,$$
	where the last equality follows from $\operatorname{Adj}(U)=\operatorname{det}(U)U^{-1}=\operatorname{det}(U)U^\top$ 
	and $\operatorname{det}^2(U)=1$ as $U$ is an orthogonal matrix.

	Given that $\Sigma\Sigma^\top=\operatorname{diag}(\sigma^2_1(B),\dots,\sigma^2_\ell (B))\in\mathbb{R}^{\ell\times \ell}$, we know from the definition of the adjugate matrix that 
	$\operatorname{Adj}(\Sigma\Sigma^\top )$ remains a diagonal matrix with all nonnegative diagonal elements. Consequently,  $\operatorname{Adj}(BB^\top)=U\operatorname{Adj}(\Sigma\Sigma^\top )U^\top$ is a positive semidefinite matrix.
	
	Since 
	$$B^\top\operatorname{Adj}(BB^\top)B
	=V\Sigma^\top U^\top \left(U\operatorname{Adj}(\Sigma\Sigma^\top )U^\top\right)U\Sigma V^\top=V\left(\Sigma^\top \operatorname{Adj}(\Sigma\Sigma^\top )\Sigma\right) V^\top,$$
	we know that 
	$$
	B^\top\operatorname{Adj}(BB^\top)B=0  \iff  \Sigma^\top \operatorname{Adj}(\Sigma\Sigma^\top )\Sigma=0
	\iff \operatorname{Adj}(\Sigma\Sigma^\top)=\operatorname{diag}(\underbrace{0,\dots,0}_{\operatorname{rank}(B)},\cdots), 
	$$
	which is equivalent to $\operatorname{rank}(B)<\ell$. This completes the proof of this lemma. 
\end{proof}

Now, we are ready to prove Lemma \ref{keylemma-exp1}.
\begin{proof}[Proof of Lemma \ref{keylemma-exp1}] 
	First, we have
	\begin{equation*}
		\begin{aligned}
			\mathbb{E} [A_{\mathcal{S} }^\dagger A_{\mathcal{S} }]
			= &\sum_{\mathcal{S}  \in \binom{[m]}{s}}\frac{\operatorname{det}(A_{\mathcal{S} }A_{\mathcal{S} }^\top)}{\sum_{\mathcal{J} \in \binom{[m]}{s}}\operatorname{det}( A_{\mathcal{J} }A_{\mathcal{J} }^\top)}A_{\mathcal{S} }^{\dagger}A_{\mathcal{S} }
			= \frac{\sum_{\operatorname{det}( A_{\mathcal{S} }A_{\mathcal{S} }^\top) \neq 0}\operatorname{det}( A_{\mathcal{S} }A_{\mathcal{S} }^\top) }{e_{s}(\lambda)}A_{\mathcal{S} }^{\dagger} A_{\mathcal{S} },
		\end{aligned}
	\end{equation*}
	where the last equality follows from Lemma \ref{lemma-spm}.
	Noting that \( \operatorname{det}(A_{\mathcal{S} } A_{\mathcal{S} }^\top) \neq 0 \) implies \( \text{rank}(A_{\mathcal{S} }) = s \), so we have $
	A_{\mathcal{S} }^{\dagger} = A_{\mathcal{S} }^\top (A_{\mathcal{S} } A_{\mathcal{S} }^\top)^{-1}.
	$
	Furthermore, it holds that
	$ \operatorname{det}(A_{\mathcal{S} } A_{\mathcal{S} }^\top)(A_{\mathcal{S} } A_{\mathcal{S} }^\top)^{-1} = \operatorname{Adj}(A_{\mathcal{S} } A_{\mathcal{S} }^\top).$
	Substituting these expressions into the above equation, we obtain
	\begin{equation}\label{proof-xie-11241}
		\begin{aligned}
			\mathbb{E} [A_{\mathcal{S} }^\dagger A_{\mathcal{S} }]&=\frac{\sum_{\operatorname{det}( A_{\mathcal{S} }A_{\mathcal{S} }^\top) \neq 0}A_{\mathcal{S} }^\top \operatorname{det}( A_{\mathcal{S} }A_{\mathcal{S} }^\top)  (A_{\mathcal{S} }A_{\mathcal{S} }^\top)^{-1}A_{\mathcal{S} }}{e_{s}(\lambda)}\\
			&= A^\top\frac{\sum_{\operatorname{det}( A_{\mathcal{S} }A_{\mathcal{S} }^\top) \neq 0}I_{\mathcal{S} }^\top \operatorname{Adj}(A_{\mathcal{S} }A_{\mathcal{S} }^\top) I_{\mathcal{S} }}{e_{s}(\lambda)}A.
		\end{aligned}
	\end{equation} 
	From Lemma \ref{lemma0}, we know that 
	$$A^\top\sum_{\operatorname{det}(A_{\mathcal{S} }A_{\mathcal{S} }^\top) = 0}I_{\mathcal{S} }^\top \operatorname{Adj}(A_{\mathcal{S} }A_{\mathcal{S} }^\top) I_{\mathcal{S} } A=\sum_{\operatorname{det}( A_{\mathcal{S} }A_{\mathcal{S} }^\top) = 0}A^\top_{\mathcal{S} } \operatorname{Adj}(A_{\mathcal{S} }A_{\mathcal{S} }^\top) A_{\mathcal{S} }=0,$$
	which together with \eqref{proof-xie-11241} implies
	$$
	\mathbb{E} [A_{\mathcal{S} }^\dagger A_{\mathcal{S} }]= A^\top\frac{\sum_{\mathcal{S}\in\binom{[m]}{s} }I_{\mathcal{S} }^\top \operatorname{Adj}(A_{\mathcal{S} }A_{\mathcal{S} }^\top) I_{\mathcal{S} }}{e_{s}(\lambda)}A
	=A^\top H_s A,
	$$
	where the last equality follows from Lemma \ref{lemmaH}. This completes the proof of this lemma.
\end{proof}

\begin{proof}[Proof of Lemma \ref{keylemma-exp2}]
	First, we have
	\begin{equation*}
		\begin{aligned}
			\mathbb{E} [I^\top_{\mathcal{S} } (A_{\mathcal{S} }^{\dagger})^\top A_{\mathcal{S} }^{\dagger}  I_{\mathcal{S} }]
			= &
			\frac{\sum_{\mathcal{S} \in \binom{[m]}{s}} I_{\mathcal{S} }^\top \operatorname{det}( A_{\mathcal{S} }A_{\mathcal{S} }^\top)(A_{\mathcal{S} }^{\dagger})^\top A_{\mathcal{S} }^{\dagger}I_{\mathcal{S} }}{\sum_{\mathcal{J} \in \binom{[m]}{s}}\operatorname{det}( A_{\mathcal{J} }A_{\mathcal{J} }^\top)} \\
			=
			&
			\frac{\sum_{\operatorname{det}( A_{\mathcal{S} }A_{\mathcal{S} }^\top) \neq 0} I_{\mathcal{S} }^\top \operatorname{det}( A_{\mathcal{S} }A_{\mathcal{S} }^\top)(A_{S}^{\dagger})^\top A_{S}^{\dagger}I_{\mathcal{S} }}{e_{s}(\lambda)}.\\
		\end{aligned}
	\end{equation*}
	Similar to the proof of Lemma \ref{keylemma-exp1}, we know that  \( \operatorname{det}(A_{\mathcal{S} } A_{\mathcal{S} }^\top) \neq 0 \) implies $
	A_{\mathcal{S} }^{\dagger} = A_{\mathcal{S} }^\top (A_{\mathcal{S} } A_{\mathcal{S} }^\top)^{-1}
	$ and \( \operatorname{det}(A_{\mathcal{S} } A_{\mathcal{S} }^\top)(A_{\mathcal{S} } A_{\mathcal{S} }^\top)^{-1} = \operatorname{Adj}(A_{\mathcal{S} } A_{\mathcal{S} }^\top) \). 
	Hence, 
	\begin{equation*}
		\begin{aligned}
			\mathbb{E} [I^\top_{\mathcal{S} } (A_{\mathcal{S} }^{\dagger})^\top A_{\mathcal{S} }^{\dagger}  I_{\mathcal{S} }]
			&=\frac{\sum_{\operatorname{det} (A_{\mathcal{S} }A_{\mathcal{S} }^\top) \neq 0} I_{\mathcal{S} }^\top \operatorname{det}( A_{\mathcal{S} }A_{\mathcal{S} }^\top)(A_{\mathcal{S} }^\top (A_{\mathcal{S} }A_{\mathcal{S} }^\top)^{-1})^\top A_{\mathcal{S} }^\top (A_{\mathcal{S} }A_{\mathcal{S} }^\top)^{-1}I_{\mathcal{S} }}{ e_{s}(\lambda)} \\
			&=\frac{\sum_{\operatorname{det}(A_{\mathcal{S} }A_{\mathcal{S} }^\top) \neq 0} I^\top_{\mathcal{S} } \operatorname{Adj}(A_{\mathcal{S} }A_{\mathcal{S} }^\top)I_{\mathcal{S} }}{e_{s}(\lambda) }.	
		\end{aligned}
	\end{equation*} 
	It follows from Lemma \ref{lemma0} that 
	$\sum_{\operatorname{det}(A_{\mathcal{S} }A_{\mathcal{S} }^\top) = 0} I_{\mathcal{S} }^\top \operatorname{Adj}(A_{\mathcal{S} }A_{\mathcal{S} }^\top)I_{\mathcal{S} } $ is positive semidefinite. Therefore,
	$$
	\mathbb{E} [I^\top_{\mathcal{S} } (A_{\mathcal{S} }^{\dagger})^\top A_{\mathcal{S} }^{\dagger}  I_{\mathcal{S} }] \preceq \frac{\sum_{ \mathcal{S} \in \binom{[m]}{s}} I^\top_{\mathcal{S} } \operatorname{Adj}(A_{\mathcal{S} }A_{\mathcal{S} }^\top)I_{\mathcal{S} }}{e_{s}(\lambda) }=H_s,
	$$
	where the last equality follows from Lemma \ref{lemmaH}. This completes the proof of this lemma.
\end{proof}

\subsection{Proof of Lemma \ref{lemma-lowerbound}}

\begin{proof}[Proof of Lemma \ref{lemma-lowerbound}]
	Let $\Delta_{s}:=\sum_{j=s}^{\operatorname{rank}(A)}\sigma^2_j(A)$ and consider the singular value decomposition  $A=U\Sigma V^\top$. It follows from Lemma \ref{lemma-xie-1125} that for any $1\leq s\leq \operatorname{rank}(A)$,
	$$H_s\succeq U\operatorname{diag}\left(\frac{1}{\sigma_1^2(A)+\Delta_{s+1}},\dots,\frac{1}{\sigma_s^2(A)+\Delta_{s+1}},\frac{1}{\Delta_{s}},\dots,\frac{1}{\Delta_{s}}\right)U^\top.$$
	Hence,
	\begin{equation}
		\begin{aligned}
			\label{eqaa1}
			A^\top H_s A&\succeq V\Sigma^\top U^\top U \operatorname{diag}\left(\frac{1}{\sigma^2_1(A)+\Delta_{s+1}},\dots,\frac{1}{\Delta_{s}},\frac{1}{\Delta_{s}},\dots,\frac{1}{\Delta_{s}}\right)U^\top U\Sigma V^\top\\
			&=V\operatorname{diag}\left(\frac{\sigma^2_1(A)}{\sigma^2_1(A)+\Delta_{s+1}},\dots,\frac{\sigma^2_s(A)}{\Delta_{s}},\dots,\frac{\sigma^2_{\min}(A)}{\Delta_{s}},0,\dots,0\right) V^\top.       
		\end{aligned}
	\end{equation}
	Since $z\in \operatorname{Range}(A^\top)$, we know that there exists a vector $ y\in\mathbb{R}^m$ such that $z=V\Sigma^\top U^\top y$. From \eqref{eqaa1}, we have
	\begin{equation}
		\begin{aligned}
			&z^\top  A^\top H_s Az \\
			&\geq (\Sigma^\top U^\top y)^\top \operatorname{diag}\left(\frac{\sigma^2_1(A)}{\sigma^2_1(A)+\Delta_{s+1}},\dots,\frac{\sigma^2_s(A)}{\Delta_{s}},\dots,\frac{\sigma^2_{\min}(A)}{\Delta_{s}},0,\dots,0\right) \Sigma^\top U^\top y   \\
			&\geq \frac{\sigma^2_{\min}(A)}{\Delta_{s}}(\Sigma^\top U^\top y)^\top(\Sigma^\top U^\top y)\\
			&=\frac{\sigma^2_{\min}(A)}{\Delta_{s}}\lVert z \rVert^2_2,
		\end{aligned}
	\end{equation}
	where the second inequality follows from $\Sigma=\operatorname{diag}(\sigma_1(A),\dots,\sigma_{\operatorname{min}}(A),0,\dots,0)\in\mathbb{R}^{m\times n}.$
\end{proof}

\end{document}